\documentclass{article}

\usepackage{amsmath,amsfonts,amssymb}
\usepackage{bbold}
\usepackage{amsthm}
\usepackage{xcolor}
\usepackage{indentfirst}

\newtheorem{monTheoreme}{Theorem}
\newtheorem*{ass}{Assumption A}
\newtheorem*{monassumption2}{Assumption B}
\newtheorem{deff}[monTheoreme]{Definition}
\newtheorem{lemma}[monTheoreme]{Lemma}

\newcommand{\pg}[1]{{#1}}

\numberwithin{equation}{section}

\begin{document}

\title{Harris's method for non-conservative periodic semiflows and application to some non-local PDEs}

\author{EL ABDOUNI Adil \footnote{Université Paris-Saclay, UVSQ, CNRS, Laboratoire de Mathématiques de Versailles, 78000, Versailles, France. Email: adil.el-abdouni@ens.uvsq.fr}}

\maketitle

\begin{abstract}
    In this paper we propose some \pg{Harris-like criteria}  in order
to study the long time behavior of general positive and periodic semiflows. These criteria allow us to obtain new existence results of principal eigenelements, and their exponential attractiveness. We present applications to two biological models in a space-time varying environment: a non local selection-mutation \pg{equation} and a growth-fragmentation \pg{equation}. The particularity of this article is to study some inhomogeneous problems \pg{that are periodic} in time, \pg{as it appears for instance} when the environment changes, \pg{due for instance to the seasonal cycle or circadian rhythms.} 
\end{abstract} \textbf{AMS Class. No.} 47A35,  35B40, 35B10, 35Q92, 45K05, 92D25, 45M05, 92D15.

\section*{Introduction}

Many phenomena in biology \pg{that} rely on periodic time scales appear in the literature.
We can cite for instance \cite{burger1995evolution,lande1996role,lynch1993evolution,lynch1991adaptive} for adaptive dynamics problems, or \cite{ Clairambault2009, clairambault2007mathematical,clairambault2006circadian} for circadian rhythms and cancer cells. Therefore, it is crucial to understand evolution equations with periodic coefficients in time, and more specifically the long-time behavior of such equations. In the field of ordinary differential equations (ODEs), the situation is well-established and extensively described by Floquet in \cite{floquet1883equations} for linear ODEs.
\pg{Regarding  linear partial differential equations (PDEs), the state of the art is much more limited.
The first eigenvalue problem is well understood and solved in~\cite{Hess1991} for general time-periodic parabolic operators, but much less is known about  non-local  equations.
As far as we know, the only existing studies consider  non-local  dispersal operators with time-periodic zero-order term, see~\cite{Hutson2008,Rawal2012,Shen2019,Sun2017} and the discussion at the end of the present introduction.}
\\

Thus, the aim of this paper is to examine the long time behavior of  non-local  linear PDEs $ \partial_t u(t) = A_t u(t) $ with a $T$-periodic operator $A_t$, which means that $A_{t+T}=A_t$ for all $t \geq 0 $. More precisely, we are searching for some conditions on the operators in order to have asynchronous Malthusian behaviour  $ u(t) \sim e^{\lambda_F(t-s)} \gamma_t \langle h_s,u_s \rangle $ when $t \rightarrow \infty$ for a fixed initial time $s$, where $u_s=u(s)$ is the initial condition. The coefficient $ \lambda_F$ is the so-called Floquet eigenvalue associated to the direct and dual periodic eigenelements $(\gamma_t)_{t \geq 0}$ and $(h_t)_{t\geq0}$. The solutions to the Cauchy problem related to a linear PDE can be represented using a  semiflow of linear operator. In this particular study, our focus lies on investigating the ergodic properties of positive semiflows.  We say that $(M_{s,t})_{t \geq s \geq 0}$ is a semiflows if $ M_{s,s}=Id $ and $M_{s,t}=M_{s,u}M_{u,t}$ for all $t\geq u \geq  s \geq 0$ , and  $T$-periodic if $M_{s+T,t+T}=M_{s,t}$ for all $t\geq s \geq 0$. In this article we consider  periodic semiflows $(M_{s,t})_{t \geq s \geq 0}$  of positive operators which act on a space of weighted signed measures on the left and on a space of weightedly bounded measurable functions on the right, and which enjoy the classical duality relation $\langle \mu M_{s,t} , f\rangle = \langle \mu , M_{s,t} f \rangle $. \\

The case of homogeneous environments, which means that $A_t=A$ for all times $t$, have been widely studied.    When the  semiflow  is conservative, which means that the total number of particles is preserved along time for a linear PDE, and more precisely $M_{s,t}1=1$ for all $ t \geq s \geq 0$, Harris's theorem provides conditions that ensure exponential convergence towards an invariant measure.  It holds when particles move without reproduction or death.
The method originates from the pioneering works of Doeblin \cite{doblin1940elements} and Harris \cite{harris1956existence}, and has been widely developed since then, see notably~\cite{canizomischler,hairer2011yet,MeynTweedie}. When delving into biological systems, the inclusion of birth and death processes frequently occurs, leading to the emergence of non-conservative semigroups. In such scenarios, the dynamics aren't anticipated to reach an invariant measure, but rather to exhibit Malthusian behavior as it has been studied in \cite{bansaye2020ergodic,bansaye2019non}.

\

Harris's method is a contraction approach that can be used for dealing with non-homogeneous environments,
 as it was successfully applied in~\cite{bansaye2020ergodic}.
For periodic environments, compared to spectral approaches, this method allows quantifying the rate of convergence to the periodic malthusian behavior.
The conditions required in~\cite{bansaye2020ergodic} are too much demanding for being applied to the equations we aim at studying in the present paper.
Our objective is thus to extend these conditions, in the spirit of~\cite{bansaye2019non}, in order to deal with periodic non-local equations which are set on unbounded domains.
More precisely, we aim at applying these modifications to two well-established models: selection-mutation and growth-fragmentation.

With that objective in mind, we draw inspiration from and employ techniques developed in the field of probability theory. These techniques serve as valuable tools to achieve our intended purpose. We refer to \cite{cloez2020irreducibility} for the establishment of the basis of this approach for homogeneous environments. 
The first part of this paper is to obtain some abstract results of general  periodic semiflows, and the second part  is devoted to applying those results  on the study of a  non-local  selection mutation  and growth fragmentation equation with two different methods. \\

\pg{Time-periodic selection-mutation equations raised interest in the last few years, and parabolic models have been considered and studied in~\cite{Carrere2020,figueroa2018long,iglesias2021selection,Lorenzi2015}.
In these models, the mutations are approximated by a local Laplace operator.
In contrast, we consider the  non-local  selection-mutation equation} \begin{align*}
    \partial_t u(t,x)+\partial_x u(t,x) &= \int_{\mathbb{R}}^{}{u(t,y)Q(y,dx)dy}+a(t,x)u(t,x),
\end{align*} \pg{which is the  non-local  counterpart of the model studied in~\cite{iglesias2021selection}.}
This equation models the dynamics of a population which is structured by a phenotypic trait $x \in \mathbb{R}$ in a time $t \geq 0$. Indeed,  the solution $u(t,x)$ can describe the density of traits $x$ in each time $t$ in a population, $a(t,x)$  models the balance between birth and death of traits $x$ in a time $t$ and we suppose that at each period of duration $T$ this fitness is repeated. The kernel $Q(y,dx)$ models the mutations through the birth of an individual with trait $x$ from one with trait $y$.
\pg{The drift term accounts for a shifting environment, that can be due for instance to global warming, see \cite{Berestycki2009}.}
We refer to the papers \cite{cloez2020irreducibility,CovilleHamel,CovilleLiWang} for the homogeneous case $a(t,x)=a(x)$.

Moreover we have a keen interest in growth models that adopt the structure of a fragmentation equation preserving mass, along with the inclusion of a transport component. These models are used to depict the progression of a population wherein each entity undergoes growth and division. The entities can be cells  \cite{Metz1986,Perthame2007}, polymers  \cite{Greer2006}, or similar entities and are characterized by a variable $x \geq 0$, which could represent size \cite{Doumic2021}, label \cite{banks2010label}, protein content \cite{doumic2007analysis}, length of a fungus filament~\cite{Tomasevi2022}, and so on.
\pg{Periodic environment appears for instance in the PMCA technique, which aims at amplifying the polymerization process of proteins by alternating incubation and sonication phases.
We refer for instance to \cite{Chyba2015,Coron2015} for more details and references.}
In order to model this, we study the following equation \begin{align*}
      &\partial_t u(t,x) +   \partial_x (\tau(t,x) u(t,x)) +\beta(t,x)u(t,x) =\int_{x}^{\infty} b(t,y,x) u(t,y)dy.  
\end{align*} The unknown $u(t,x)$ represents the population density at time $t$ of some particles with size $x>0$. The fragmentation kernel is of the form $b(t,x,y)= \frac{1}{x}\kappa\left( \frac{y}{x} \right) \beta(t,x)$ where $ \kappa$ is the fragmentation distribution. Each particles grows with speed $ \tau(t,x)$ and splits with rate $  \beta(t,x) $, \pg{and these rates are supposed to be periodic in time}.
\\

The following models, closely related to our models, have been widely studied in a series of paper \cite{Hutson2008,Rawal2012,Shen2019,Sun2017}.
\[-\partial_tu(t,x) + \int_\Omega K(x,y)u(t,y)\,dy + a(t,x)u(t,x) = \lambda u(t,x),\quad u(t+T,\cdot)=u(t,\cdot),\]
where $x\in\Omega$, with $\Omega\subset \mathbb{R}^N$ a bounded domain, and $a(t+T,\cdot)=a(t,\cdot)$ for all $t\geq0$. Indeed they obtain the existence of a principal eigenfunction in different cases. In~\cite{Hutson2008}, the authors treat the one dimensional case $N=1$ for $K>0$ continuous and $a$ Lipschitz. The higher dimension $N \in \mathbb{N}$ is considered in~\cite{Rawal2012}, where sufficient conditions are provided in the case of a convolutive kernel $K(x,y)=J(x-y)$ with $J\in C^1$ and $J(0)>0$. These conditions were then sharpened and somehow simplified by~\cite{Shen2019}. In~\cite{Sun2017}, the authors are interested in the case where $a=0$ on a subdomain $\Omega_0\subset\Omega$, still for any $N \in \mathbb{N}$, and a kernel $K(x,y)=J(x-y)$ with $J$ continuous and even. The main differences compared to our examples is that in those works there is no transport term involved, the domain is bounded, and the integral kernel is always independent of time, while it is time periodic in our growth-fragmentation example.
However, the dimension of the trait space can be higher than 1 in \cite{Rawal2012,Shen2019,Sun2017}, and in order to apply our method to our models, we will need to make some assumptions on the kernel and fitness terms which will be introduced in Sections 2 and 3.

\

The organization of the paper is the following. In  Section 1 , we present a contraction method and propose two sets of assumptions allowing to obtain a convergence result presented by Theorem 1. We use Assumption~B for the selection-mutation model that will be introduced in Section 2. Finally, we obtain a convergence result of a  growth-fragmentation model in Section 3, by using Assumption A.

\section{An abstract result}

We are interested in the long time behavior of positive and periodic semiflows in weighted signed measures spaces. Let us start by defining more precisely what we mean by weighted signed measures. Let $ X$ be a measurable space, for a function $V:X \rightarrow (0,\infty)$ we denote by $\mathcal{M}_+(V)$ the set of positive measures on $X$ such that  \begin{align*}
\forall \mu \in \mathcal{M}_+(X), \hspace{0.2cm} \langle \mu, V \rangle =\int_{ X }^{} V d \mu < \infty.
\end{align*}We define the space of weighted signed measures as \begin{align*}
\mathcal{M}(V)=\mathcal{M}_+(V)-\mathcal{M}_+(V).\end{align*}  It is the smallest vector space with positive cone $\mathcal{M}_+(V)$,  i.e.  an element $\mu$  of $\mathcal{M}(V)$ is the difference of two positive measures $\mu_+, \mu_- \in \mathcal{M}_+(V)$ which are mutually singular (by the  Hahn-Jordan  decomposition). For a rigorous construction we refer to Section 2 of \cite{bansaye2019non}. As we are going to use a duality approach by working on functions, we add that  an element from $ \mathcal{M}(V)$   acts on the Banach space  \begin{align*}
    \mathcal{B}(V)=\left\{ f:X \rightarrow \mathbb{R} \hspace{0.1cm} \text{measurable}, \hspace{0.2cm} \| f  \|_{\mathcal{B}(V)}=\sup_{x \in X} \frac{|f(x)|}{V(x)} < \infty \right\},\end{align*} through the following linear relation \begin{align*}
     \langle \mu ,f \rangle = \int_{ X }^{} f d \mu_+ - \int_{ X }^{} f d \mu_-, \quad \text{when} \hspace{0.1cm} \mu \in \mathcal{M}(V),f\in \mathcal{B}(V).    \end{align*}  The space $\mathcal{M}(V)$ is a Banach space for the weighted total variation norm \begin{align*}
\| \mu\|_{\mathcal{M}(V)}=  \underset{ \| f\|_{\mathcal{B}(V)} \leq 1}{\sup} \langle \mu ,f \rangle.     \end{align*}  In the case $V=1$ we will sometimes denote $\mathcal{M}(X)$ and $\mathcal{B}(X)$ instead of $ \mathcal{M}(1) $ and $\mathcal{B}(1)$. \\  Let us  recall a result proved in \cite{bansaye2019non} for a positive discrete time semigroup $(M_{n })_{n \in \mathbb{N}}=(M^n)_{n \in \mathbb{N} }$ by first writing the assumptions necessary for the result.
\makeatletter\tagsleft@true\makeatother  

\begin{ass}
There exist a function $\psi:X \rightarrow (0,\infty)$ with $\psi \leq V$, some integer $k>0$, real numbers $\beta > \alpha >0$, $\theta >0$, $(c,d) \in (0,1]^2$, a subset $K \subset X$ such that $\sup_K V / \psi < \infty$, and some probability measure $\nu $ on $ X $ suported by $K$ such that 

\begin{align} 
    M^k V \leq \alpha V + \theta \mathbb{1}_K \psi \tag{A1}\\
    M^k \psi \geq \beta \psi \tag{A2} \\
    \forall f \in  \mathcal{B}_+(V/\psi), \hspace{0.1cm} \underset{x \in K}{\inf } \frac{M^k(f \psi)(x)}{M^k \psi(x)} \geq c \langle \nu, f \rangle \hspace{0.2cm}  \tag{A3}\\
     \forall n \in \mathbb{N}, \hspace{0.1cm} \left \langle \nu , \frac{M^{nk} \psi}{\psi} \right \rangle \geq d \underset{x \in K}{\sup}\frac{M^{nk}\psi(x)}{\psi(x)} \tag{A4}.
\end{align}

\end{ass}  We note that in the conservative case, meaning that the operator $M$ verifies $M1 =1 $, the conditions (A2) and (A4) are satisfied with $\beta=d=1$ and $\psi= 1$, (A1) reads $M^k V \leq \alpha V + \theta \mathbb{1}_K$  and (A3) can be written $M^kf(x) \geq c \langle \nu , f\rangle$ for all $x \in K$ and for all $f \geq 0$.
We thus recover for stochastic semigroups the classical Harris theorem and the set of conditions in Assumption~A is an extension to the non-conservative case: (A1) is a Foster-Lyapunov condition, (A2) relaxes the conservativity assumption of Harris's theorem, (A3) is a minoration condition and (A4) is an additional assumption required for dealing with non-conservative operators.
Roughly speaking, the condition (A4) guarantees that the non-conservativeness of the semigroup does not breaks down the contraction property gained by the minoration condition.  With those assumptions, \cite[Theorem 5.5]{bansaye2019non} reads as follows.
 
\begin{monTheoreme}
\label{yo}
Let $(V, \psi)$ be a couple of measurable functions from $ X $ to $(0,\infty)$ such that $\psi \leq V$ and which satisfies Assumption A. Then there exist a unique triplet $(\gamma, h , \Lambda) \in \mathcal{M}_+(V) \times \mathcal{B}_+(V) \times \mathbb{R}$ of eigenelements of M with $\langle \gamma , h \rangle = \| h \|_{\mathcal{B}(V)}=1$ and $h >0$ i.e. satisfying \begin{align*}
    \gamma M = \Lambda \gamma \hspace{0.5cm} \text{and} \hspace{0.5cm} Mh= \Lambda h.
\end{align*} Moreover, there exists $C>0$ and $\omega >0$ such that for all $n \geq 0$   and all $\mu \in \mathcal{M}(V)$ .\begin{align*}
     \|  \Lambda^{-n} \mu M^n - \langle \mu , h \rangle \gamma \|_{\mathcal{M}(V)} \leq C e^{-n \omega}  \| \mu \|_{\mathcal{M}(V)}.
 \end{align*}
\end{monTheoreme}   The first aim of the present paper is to derive from Theorem \ref{yo} a similar result for time continuous periodic semiflows.  We first introduce a generalisation of eigenelements for periodic semiflows. We say that $(\gamma_{s} , h_{s}, \lambda_F)_{s \geq 0}$ is a $T$-periodic Floquet family for a $T$-periodic semiflow   $(M_{s,t})_{t \geq s \geq 0}$ if for all $t \geq s \geq 0$  \begin{align*}
\gamma_{s+T}=\gamma_{s} \; \; &\text{and}  \; \; h_{t+T}=h_{t}. \\
 \gamma_{s}M_{s,t}=e^{\lambda_F(t-s)}\gamma_{t} \; \;&\text{and} \; \; M_{s,t} h_{t} =e^{\lambda_F(t-s)}h_{s}.
\end{align*} 

\makeatletter\tagsleft@false\makeatother 

\begin{monTheoreme}
\label{thm1.2}

Let $(M_{s,t})_{0 \leq s \leq t}$ be a positive T-periodic semiflow such that $M_{s_0,s_0+T}$ satisfies Assumption A for some functions $V \geq \psi > 0$ and some $s_0 \in [0,T)$.  We also suppose that $(s,t) \rightarrow \|M_{s,t}V \|_{\mathcal{B}(V )}$ is
locally bounded.  Then there exists a unique T-periodic Floquet family $(\gamma_{s} , h_{s}, \lambda_F)_{s \geq 0}\subset \mathcal{M}_+(V) \times \mathcal{B}_+(V)\times \mathbb{R}$ such that  $ \langle \gamma_{s} , h_{s} \rangle  = 1$ for all $s \geq 0$ and $\| h_{s_0} \|_{\mathcal{B}(V )}=1$,  and there exist $C \geq 1$, $\omega > 0$ such that for all $t \geq s \geq 0$ and all $\mu \in \mathcal{M}(V)$ 
\begin{align}\label{innequa}
    \| e^{-\lambda_F(t-s)} \mu M_{s,t} - \langle \mu ,h_{s} \rangle \gamma_{t} \|_{\mathcal{M}(V)} \leq Ce^{-\omega(t-s)} \|\mu  \|_{\mathcal{M}(V)}. 
\end{align}

\end{monTheoreme} The result of this theorem means that for all initial data $\mu$, the trajectory on the semiflow behaves asymptotically as
$$\mu M_{s,t} \sim  \langle \mu,h_s \rangle e^{\lambda_F(t-s)} \gamma_t\qquad\text{when}\ t\to+\infty,$$
that is, an exponential growth (or decay) with parameter $\lambda_F$ aligning with a periodic profile $\gamma_t$.  Both $\lambda_F $ and $\gamma_t$ are universal, meaning they do not depend on the initial data $\mu$. The trace of the initial data is given by $h_s$ through the scalar term  $\langle \mu , h_s  \rangle $.

\begin{proof}

Applying Theorem \ref{yo} to the operator $M=M_{s_0,s_0+T}$, we get the existence of a unique triplet $(\gamma, h , \Lambda) \in \mathcal{M}_+(V) \times \mathcal{B}_+(V) \times \mathbb{R}$ of eigenelements of $M$ with $\langle \gamma , h \rangle = \| h \|_{\mathcal{B}(V)}=1$ and $h>0$, which satisfies \begin{align*}
    \gamma M_{s_0,s_0+T} = \Lambda \gamma \hspace{0.5cm} \text{and} \hspace{0.5cm} M_{s_0,s_0+T}h=\Lambda h.
\end{align*} For $k \in \mathbb{Z}$, $s \in [s_0+kT,s_0+(k+1)T] $ and $t \in [ s_0+(k-1)T,s_0+kT]$ we define $ \lambda_{F} = \frac{\log(\Lambda)}{T}$ and \begin{align*}
   h_{t}=e^{-\lambda_F (s_0+kT-t)}M_{t,s_0+kT}h   &\hspace{0.5cm} \gamma_{s}=e^{\lambda_F (s_0+kT-s)}\gamma M_{s_0+kT,s}.
\end{align*}Note that we have in particular $h_{s_0+kT}=h$ and $\gamma_{s_0+kT}=\gamma$ for all $k \in \mathbb{Z}$. Let us take $s\in (s_0+kT,s_0+(k+1)T)$, we will show that $\langle \gamma_{s} , h_{s} \rangle = \| h_{s} \|_{\mathcal{B}(V)}=1$, $h_{s+T}=h_{s}$  and  $\gamma_{t+T}=\gamma_{t}.$ \begin{align*}
    \langle \gamma_{s} , h_{s} \rangle &=  e^{- \lambda_F(s-s_0-kT)} \langle \gamma M_{s_0+kT,s}, h_{s} \rangle =  e^{-\lambda_F(s-s_0-kT)} \langle \gamma,M_{s_0+kT,s} h_{s} \rangle =e^{-\lambda_F T}\langle \gamma, M_{s_0+kT,s_0+(k+1)T} h \rangle\\
    &=\Lambda e^{-\lambda_F T} \langle \gamma  , h \rangle =1,\\
    h_{s+T}&=e^{-\lambda_F (s_0+(k+2)T-s-T)}M_{s+T,s_0+(k+2)T}h=e^{-\lambda_F (s_0+(k+1)T-s)}M_{s,s_0+(k+1)T}h=h_s,\\
    \gamma_{s+T}&=e^{\lambda_F (s_0+(k+1)T-s-T)}\gamma M_{s_0+(k+1)T,s+T}=e^{\lambda_F (s_0+kT-s)}\gamma M_{s_0+kT,s}=\gamma_s.
\end{align*} So we proved that $(\gamma_{s},h_{s},\lambda_F)_{s \geq 0}$ is a $T$-periodic Floquet family, and we will now prove the stability estimate (\ref{innequa}). We know from Theorem $1$ that there exist $C>0$ and $\omega >0$ such that for all $n \geq 0$  and for all $\mu \in \mathcal{M}(V)$ \begin{align*}
     \|  \Lambda^{-n} \mu M^n_{s_0,s_0+T} - \langle \mu , h \rangle \gamma \|_{\mathcal{M}(V)} \leq C e^{-n \omega}  \| \mu \|_{\mathcal{M}(V)}.
 \end{align*} 
 
 So, for $s=s_0$ and $t=s_0+nT$ equation (\ref{innequa}) holds. Let us take two integers $n,k \geq 0$ such that $0\leq s_0+kT-s<T$ and $0 \leq t-s_0-(k+n)T<T $ , which gives the existence of  $t'\in(0,2T)$  such that $t-s=nT+t'$. Moreover $(s,t) \rightarrow \|M_{s,t}V \|_{\mathcal{B}(V )}$ is locally bounded, so there exist $\widetilde{C}>0$ such that  \begin{align*}
  \max \left( \|M_{s,s_0+kT}V \|_{\mathcal{B}(V )}, \|M_{s_0+(k+n)T,t}V \|_{\mathcal{B}(V )} \right) \leq \widetilde{C}. 
\end{align*}  This gives
\begin{align*}
\| e^{-\lambda_F(t-s)} \mu M_{s,t} - \langle \mu ,h_{s} \rangle \gamma_{t} \|_{\mathcal{M}(V)} &= \| e^{-\lambda_F(nT+t')} \mu M_{s,s_0+kT} M^n_{s_0,s_0+T} M_{s_0+(k+n)T,t} - \langle \mu ,h_{s} \rangle \gamma_{t} \|_{\mathcal{M}(V)} \\
&=\| e^{-\lambda_F t'}\left(\Lambda^{-n}\mu M_{s,s_0+kT} M^n_{s_0,s_0+T} - \langle \mu M_{s,s_0+kT} ,h \rangle  \gamma \right) M_{s_0+(k+n)T,t} \|_{\mathcal{M}(V)}\\
 &\leq  \widetilde{C} C e^{-n \omega}e^{-\lambda_Ft'} \| \mu M_{s,s_0+kT}  \|_{\mathcal{M}(V)} \\
  &\leq  C'  e^{-\omega(t-s)/T}  \| \mu \|_{\mathcal{M}(V)}, 
\end{align*}  with  $ C'= \widetilde{C}^2 C e^{2(|\lambda_F|T+ \omega)}$, where we have used the definition of $t'$ in the last inequality , and this concludes the proof.
\end{proof} In practice the assumption (A4) might be quite complicated to verify, so we will replace it with  other  assumptions that are stronger but sometimes easier to check.
For some functions $f,g : \Omega \rightarrow \mathbb{R} $, the notation $f \lesssim g $ means that there exists a constant $C>0$ such that $f(x) \leq C g(x)$ for all $x \in \Omega$.
 \makeatletter\tagsleft@true\makeatother
\begin{monassumption2}
There exist a time $s_0 \geq 0$, a subset $K \subset X$, a time $\tau > 0 $, constants $\beta > \alpha >0$, $c\in (0,1]$ and $\theta>0$ and some probability measure $\nu $ on  $X$  supported by $K$ such that for a couple of functions $(V,\psi)$ from $X$ to $(0,\infty)$ which verifies $\psi \leq V $ on $X$ and $V \lesssim \psi $ on $K$ we have \begin{align} 
       M_{s,t}V \lesssim V \hspace{0.1cm } \text{and} \hspace{0.1cm} M_{s,t} \psi \gtrsim \psi \hspace{0.1cm} \text{on }  \hspace{0.1cm}  X \hspace{0.1cm}  \text{uniformly over} \hspace{0.1cm} s_0 \leq s \leq t \leq s_0+2\tau\tag{B0} \\
        \forall s \geq 0, \hspace{0.1cm} M_{s,s+\tau} V \leq \alpha V + \theta \mathbb{1}_K \psi  \tag{B1} \\
        \forall s \geq 0, \hspace{0.1cm}  M_{s,s+\tau} \psi \geq \beta \psi  \tag{B2}\\
        \forall f \in  \mathcal{B}_+(V/\psi) , \hspace{0.1cm} \underset{x \in K}{\inf } \frac{M_{s_0,s_0+\tau}(f \psi)(x)}{M_{s_0,s_0+\tau} \psi(x)} \geq c \langle \nu, f \rangle \tag{B3}\\
        \exists C>0, \; \forall (s,n,y) \in [s_0,s_0+\tau]\times \mathbb{N} \times K \quad \frac{M_{s_0,s_0+n\tau}\psi(y)}{M_{s,s+n\tau}\psi(y)} \leq C,  \tag{B4}
\end{align} and there exist $d > 0$ and a family of probability measures ($\sigma_{x,y})_{x,y \in K}$ over $[0,\tau]$ such that \begin{align}
     \forall f \in \mathcal{B}_+(V), \hspace{0.1cm} \forall x,y\in K,  \hspace{0.1cm}\frac{M_{s_0,s_0+\tau } f(x)}{\psi(x)} \geq d \int_{s_0}^{s_0+\tau} \frac{M_{u,s_0+\tau}f(y)}{\psi(y)}\sigma_{x,y}(du).  \tag{B5}
\end{align}
\end{monassumption2} \makeatletter\tagsleft@false\makeatother 
Condition (B5) was introduced in \cite{cloez2020irreducibility} in the non-periodic setting. In this work, it was shown that this condition holds when there exists a time $\tau>0$ such that, with (uniform with respect to $x,y\in K$) positive probability, trajectories issued at time $s_0$ from $x$ intersect at time $s_0+\tau$ the trajectories issued from $y$ at some random times $u\in [s_0,s_0+\tau]$.
Conversely, if we consider the conservative case $M_{s,t}1=1$ and $\psi=1$, then the semiflow is associated to a Markov process and the condition~(B5) ensures the existence of a coupling between the law associated with $x,s_0, \tau$ and the law associated with $y$ and a random time in $[s_0,s_0+\tau]$ in such a way the random variables are equal with probability $d$.

We present an example which explain why we need to add the condition (B4) unlike the homogeneous case presented in the paper \cite{cloez2020irreducibility} which not need (B4). Let us consider the equation
\[
 \partial_t u(t,x)+\partial_x u(t,x)=\sin(t-x)u(t,x), \qquad (t,x) \in \mathbb{R}^2. 
\]

We define the associated semiflow $(M_{s,t})_{t \geq s \geq 0}$, that is the semiflow such that the solution $u$ with initial condition $u(s,\cdot)=u_s$ is given by $u(t,\cdot)=u_sM_{s,t}$.
To obtain an explicit expression of the right action of this semiflow, we use that the function $\varphi(s,x)=M_{s,t}f(x)$ satisfies the backward dual equation
\[-\partial_s\varphi(s,t)-\partial_x\varphi(s,x)=\sin(s-x)\varphi(s,x)\]
with terminal condition $\varphi(t,x)=f(x)$.
We refer to Section~\ref{sec:selec-mut} for a rigorous justification of this fact.
Using the characteristics' method, we readily obtain the formula $M_{s,t}f(x)=f(x+t-s)e^{(t-s)\sin(s-x)}$. 
A simple computation shows that \begin{align*}
\frac{M_{s,s+kT}\psi(x)}{M_{u,u+kT} \psi(x)}=e^{k T \left(  \sin(s-x)-\sin(u-x)  \right)}.
\end{align*}   So, we clearly see that if $ \sin(s-x) > \sin(u-x) $, then $ \lim\limits_{k \rightarrow +\infty} \frac{M_{s,s+k T}\psi(x)}{M_{u,u+k T} \psi(x)}=+\infty$.

Now,  we  will prove that Assumption B implies Assumption A for $M_{s_0,s_0+T}$, by showing that Assumption B implies the existence of a constant $d>0$ such that for all $x,y \in K$ and all $n\geq 0$ \begin{align} \label{yota}
    \frac{M_{s_0,s_0+n\tau}\psi(x)}{\psi(x)} \geq d \frac{M_{s_0,s_0+n\tau} \psi(y)}{\psi(y)}.  
\end{align} It will prove (A4), because $\nu$ is a probability and $\psi$ is bounded on $K$.  Before  starting  the proof we begin with a simple lemma  which shows that we can control the ratio $\frac{M_{s,t}\psi}{M_{s,u}\psi}$ for fixed initial time $s$, uniformly in the final times $t$ and $u$ with $|t-u|$ bounded.

 \begin{lemma}\label{lemma imp} We take a $T$-periodic semiflow $(M_{s,t})_{0 \leq s \leq t }$ which verifies $(B0)$, $(B1)$ and $(B2)$. Then, there exist $C,c_1,c_2>0$ such that for any $s \geq 0$  and  $t \geq u \geq s $ such that $t-u< \tau $, we have for all $x \in K$

 \begin{align*}
  M_{s,t}V(x) \leq C M_{s,t}\psi(x) \\
 c_1 M_{s,u}\psi(x)\leq  M_{s,t}\psi(x) \leq c_2 M_{s,u}\psi(x).
 \end{align*}

\end{lemma}

\begin{proof}
Let us fix $t\geq s$ and take $k \in \mathbb{N}$ such that $ t \in [ s+k  \tau ,s+(k+1) \tau  ] $. Let us start by proving  that there exists $C>0$ such that  $ M_{s,t}V(x) \leq C M_{s,t}\psi(x) $  for all $x \in K$. Using assumptions (B1) and (B2) we have \begin{align*}
     M_{s,s+k  \tau  }V \leq \alpha M_{s,s+(k-1)  \tau  }V +\theta M_{s,s+(k-1)  \tau  }\psi; \hspace{0.2cm}       M_{s,s+k  \tau } \psi \geq   \beta M_{s,s+(k-1) \tau } \psi,
\end{align*} which yields \begin{align*}
    \frac{ M_{s,s+k  \tau  }V }{  M_{s,s+k  \tau } \psi} \leq \frac{\alpha}{\beta} \frac{M_{s,s+(k-1) \tau  }V}{M_{s,s+(k-1) \tau } \psi}+ \theta,
\end{align*}  and we deduce by induction \begin{align*}
    \frac{ M_{s,s+k \tau }V }{  M_{s,s+k \tau } \psi}   \leq \left(\frac{\alpha}{\beta} \right)^k \frac{V}{\psi} + \frac{\theta}{\beta-\alpha}. 
\end{align*} Using (B0) and the fact that $\frac{\alpha}{\beta}\leq 1$, we finally get \begin{align} \label{1.4}
  \frac{M_{s,t}V}{M_{s,t}\psi} \lesssim   \frac{ M_{s,s+(k-1)  \tau }V }{  M_{s,s+(k-1)  \tau } \psi}  \lesssim  \frac{V}{\psi}.
\end{align} uniformly on $t \geq s \geq 0$. Now we prove that  $M_{s,t}\psi \lesssim  M_{s,u}\psi$ in $K$, uniformly in $u \in [t- \tau ,t] $. To do so, we write \begin{align*}
M_{s,t}\psi\leq M_{s,t}V =   M_{s,u}  M_{u-k  \tau ,t-k \tau  } V \lesssim M_{s,u}V \lesssim M_{s,u}\psi, 
\end{align*} where we have used (\ref{1.4}) in the last inequality.  To finish, we prove $M_{s,u}\psi \lesssim M_{s,t}\psi$, which readily follows from (B0) by writing \begin{align*}
 M_{s,t}\psi =   M_{s,u}  M_{u-k  \tau ,t-k  \tau } \psi \gtrsim M_{s,u}\psi.
\end{align*}
\end{proof}Now we have the tools to demonstrate that Assumption B implies Assumption A for $M_{s,s+T}.$\begin{monTheoreme}\label{periodicité}
  Suppose that the $T$-periodic semiflow $(M_{s,t})_{0 \leq s \leq t}$ verifies Assumption B for a time $s_0 \geq 0$ and $\tau=kT$ for an integer $k$. Then the operator $M_{s,s+T}$ verifies Assumption A.
\end{monTheoreme}
\begin{proof} Let us fix $s_0 \geq 0 $. By Assumption (B0), Condition (\ref{yota}) is clearly verified for $t \in [s_0,s_0+\tau]$. When $t > s_0+\tau $, using $f = M_{s_0+\tau,t} \psi(x) $ in (B5) ensures that for all $x,y \in K$ \begin{align*}
   \frac{M_{s_0,t} \psi(x)}{\psi(x)} \geq c \int_{s_0}^{s_0+\tau} \frac{M_{u,t}\psi(y)}{\psi(y)}\sigma_{x,y}(du)  .
\end{align*} So, in order to show (\ref{yota}), we need to verify \begin{align} \label{astq}
    \exists C>0, \quad \forall (u,y) \in [s_0,s_0+\tau]\times K, \quad  \frac{M_{s_0,t} \psi(y) }{ M_{u,t} \psi(y)} \leq C .
\end{align}  For $u\in [s_0,s_0+\tau]$, let us take $l=\lfloor \frac{t-s_0}{\tau} \rfloor$, so that $s_0 \leq t-l\tau < s_0+\tau$ and $ u \leq t-(l-1)\tau \leq s_0+2\tau $. By Lemma \ref{lemma imp}, we have \begin{align*}
\frac{M_{s_0,t}\psi}{M_{u,t}\psi}  \lesssim \frac{M_{s_0,s_0+l\tau}\psi}{M_{u,u+l\tau}\psi}.
\end{align*} By using (B4) we clearly obtain (\ref{astq}). Moreover, we note that Assumption B  clearly implies (A1), (A2) and (A3), which concludes the proof.

\end{proof}

\section{Application to a selection-mutation model}\label{sec:selec-mut}

We apply our method to the following non-local equation with drift 
\begin{equation}\label{yop}
\left\{\begin{array}{l}
    \displaystyle\partial_t u(t,x)+\partial_x u(t,x) = \int_{\mathbb{R}}^{}{u(t,y)Q(y,dx)dy}+a(t,x)u(t,x), \quad (t,x)\in (s,\infty) \times \mathbb{R} \vspace{2mm} \\
    u(s,x)=u_s(x). 
\end{array}\right.
\end{equation}
We suppose that $a: \mathbb{R}^2 \rightarrow \mathbb{R}$ is a continuous function on $\mathbb{R}^2$, $T$-periodic in time and $C^1$ in space such that 
\begin{align} \label{a1}
   \forall R>0, \exists C_R \in \mathbb{R},  \forall (x,t,\alpha) \in [-R,R] \times \mathbb{R} \times [0,T],  \int_{0}^{t}a(\tau+\alpha,x+\tau)d\tau \leq C_R+ \int_{0}^{t}a(\tau,x+\tau)d\tau, 
\end{align} and there exist  continuous functions $\underline{a},\overline{a}$  with $\underline{a} \leq 0$  which verify $\lim\limits_{ |x| \rightarrow \infty} \overline{a}(x)=-\infty$ and $\underline{a}(\cdot) \leq a(t,\cdot)   \leq \overline{a}(\cdot) \hspace{0.1cm} \text{for all  } \hspace{0.1cm} t \in \mathbb{R}_+ ,$  and we set $ A= \sup_{ \mathbb{R}} \overline{a} $. A function which verifies (\ref{a1}) can be for example $a(t,x)= -\displaystyle\sqrt{  |x+\sin(t) | } $. 
    
     For all $x\in \mathbb{R}$, $Q(x,\cdot)$ is a finite positive measure on $\mathbb{R}$. We write (abusively) $Q(y,dx)$ to means that $Q(y,\cdot)$ is a measure in the $x$ variable. We assume that $x \mapsto Q(x,\cdot)$ is a narrowly continuous function $\mathbb{R} \rightarrow \mathcal{M}_+(\mathbb{R})$ which satisfies \begin{align}
     \exists \epsilon, \kappa_0, \forall x \in \mathbb{R} ;\ Q(x, \cdot) \geq \kappa_0 \,\mathcal{U}_{x,\epsilon} \label{Q1}  \\
     \widehat{Q}=\sup_{x \in \mathbb{R}}Q(x,\mathbb{R}) < \infty   \label{Qpasut} \\
      \exists C_1 \in \mathbb{R}, ;\ \forall (x,\alpha) \in \mathbb{R} \times [0,T] ;\  Q(x+\alpha,\cdot) \leq C_1Q(x,\cdot) . \label{Q12}
\end{align} where $\mathcal{U}_{x,\epsilon}$ is the (non-normalized) uniform measure on $[x-\epsilon,x+\epsilon]$, namely the measure defined for any measurable subset $A$ of~$\mathbb{R}$ by $\mathcal{U}_{x,\epsilon}(A)=\mathcal{L}\left( A\cap [x-\epsilon,x+\epsilon] \right)$ with $\mathcal{L}$ the Lebesgue measure.  We recall that $ (\mu_n)_{n \geq 0}$ converges narrowly to $\mu$ if $ \int_{}^{}f d\mu_n $ converges to  $ \int_{}^{}f d\mu  $ for every bounded function $f$. An example of a  kernel    verifying ($\ref{Q12}$) can be given by   $Q(x,dy)=e^{-|x-y|}dy$.

\begin{monTheoreme}\label{ntm}
Under Assumptions (\ref{a1}), (\ref{Q1}), (\ref{Qpasut}) and (\ref{Q12}), there exist constants $C,\omega > 0$ and a unique T-periodic Floquet family  $(\gamma_{s} , h_{s}, \lambda_F) \subset  \mathcal{M}_+(\mathbb{R}) \times \mathcal{B}_+(\mathbb{R}) \times \mathbb{R}$ with  $\langle \gamma_{s}, h_{s} \rangle=1$ for all $s \geq 0$ and $   \left\| h_{s_0 }  \right\|_{\infty}  =1$ for a certain $s_0 \geq 0$ , such that for any initial condition $u_s \in \mathcal{M}(\mathbb{R})$ the corresponding solution $u(t,x)$ of equation (\ref{yop}) verifies for all $t \geq s$

$$ \left\| u(t,.)e^{-\lambda_F(t-s)}-\left(\int_{\mathbb{R}}^{}{h_{s}(x) u_s(dx)}\right)\gamma_{t} \right\|_{TV} \leq C \left\|  u_s \right\|_{TV} e^{-\omega (t-s)}.$$
\end{monTheoreme} The method of proof consists in applying our general result, Theorem \ref{thm1.2}, to the semiflow $(M_{s,t})_{t \geq s \geq 0}$ associated to Equation (\ref{yop}), which is defined through the Duhamel formula \begin{align} \label{2.5}
     M_{s,t}f(x)=f(x+t-s)e^{ \int_{s}^{t}{a(\tau,x+\tau-s)d\tau}} + \int_{s}^{t}{ e^{\int_{s}^{\tau}{a(\tau',x+\tau'-s)d\tau'}} \int_{\mathbb{R}}^{}{M_{\tau,t}f(y)Q(x+ \tau-s,dy)}d\tau}
\end{align} by using the Banach fixed point theorem. This semiflow satisfies \begin{align*}
\partial_tM_{s,t}f=M_{s,t}\mathcal{L}_tf \hspace{0.2cm} \text{and} \hspace{0.2cm} \partial_sM_{s,t}f=-\mathcal{L}_sM_{s,t}f,
\end{align*} where $ \mathcal{L}_s$ is the generator associated to equation $(2.1)$, which reads \begin{align*}
\mathcal{L}_sf(x)=f'(x)+a(s,x)f(x)+\int_{\mathbb{R}}^{}f(y)Q(x,dy).
\end{align*} We will rigorously demonstrate that all is well defined in some suitable spaces.
 
\subsection{The dual equation}
 
 Let us start by setting the dual equation of (\ref{yop}) which is \begin{align} \label{2.6}
    -\partial_s \psi(s,x) - \partial_x \psi(s,x) = \int_{\mathbb{R}}^{}{\psi(s,y)Q(x,dy)+a(s,x) \psi(s,x)}. 
\end{align} We denote $Y = \left\{ (s,t,x) \in \mathbb{R}^3, \hspace{0.2cm}  0 \leq s \leq t \right\} $   and for all $L>0$ we define $ Y^L= \left\{ (s,t,x) \in \mathbb{R}^3, \ \max(0,t-L) \leq s \leq t \right\} $ and  $\mathcal{B}_{\text{loc}}=  \left\{ f: Y \rightarrow \mathbb{R} \hspace{0.1cm}\text{measurable}, \quad f_{|Y^L} \in \mathcal{B}(Y^L) \ \text{for all} \  L>0 \right\}$.  \begin{deff}
We say that $\psi \in  \mathcal{B}_{\mathrm{loc}}(Y)$ is solution of (\ref{2.6}) with final condition $\psi(t,t,\cdot)=f$ when for all $ (s,t,x) \in Y$  \begin{align} \label{2.7}
  \psi(s,t,x)=f(x+t-s)e^{ \int_{s}^{t}{a(\tau,x+\tau-s)d\tau}} + \int_{s}^{t}{ e^{\int_{s}^{\tau}{a(\tau',x+\tau'-s)d\tau'}} \int_{\mathbb{R}}^{}{\psi(\tau,t,y)Q(x+ \tau-s,dy)}d\tau}.  
\end{align}

\end{deff} Let's start by proving the following theorem using Banach's fixed point theorem.

\begin{monTheoreme}
For all $f \in \mathcal{B}(\mathbb{R})$, there  exists  a unique function $\psi \in  \mathcal{B}_{\mathrm{loc}}(Y)$  solution to (\ref{2.6}) with final condition $\psi(t,t,\cdot)=f$. Additionally if $f \in C_b(\mathbb{R})$ then $\psi \in  C(Y) \cap  \mathcal{B}_{\mathrm{loc}}(Y) $.  Moreover, when $f \in C_b^1(\mathbb{R})$ we have \begin{align*}
\partial_tM_{s,t}f=M_{s,t}\mathcal{L}_tf \hspace{0.2cm} \text{and} \hspace{0.2cm} \partial_sM_{s,t}f=-\mathcal{L}_sM_{s,t}f,
\end{align*}  where $M_{s,t}f(x)= \psi(s,t,x)$ for all $t \geq s \geq 0$ and $x \in \mathbb{R}$.
\end{monTheoreme}

\begin{proof}
Let  $\epsilon>0$ to be chosen later and denote for all $k \in 
[\![0;n]\!]$ \begin{align*}
 Y_{\epsilon}^k= \left\{  (s,t,x) \in \mathbb{R}^3, \hspace{0.2cm} \max(0,t-(k+1)\epsilon) \leq s \leq  t-k\epsilon  \right\}.
\end{align*}  We will construct by induction  a function $\psi_k$ in each set $Y_{\epsilon}^k$. Let $f_0 \in  \mathcal{B}(\mathbb{R})  $, we start by defining the operator $\Gamma_0: \mathcal{B}( Y_{\epsilon}^0)  \rightarrow \mathcal{B}(Y_{\epsilon}^0)  $ by \begin{align*}
    \Gamma_0 g(s,t,x) &=f_0(x+t-s)e^{\int_{s}^{t}a(\tau,x+\tau-s)d\tau}+ \int_{s}^{t}e^{\int_{s}^{\tau}a(\tau',x+\tau'-s)d\tau'} \int_{\mathbb{R}}^{}g(\tau,t,y)Q(x+\tau-s,dy)d\tau .
\end{align*} For $g_1, g_2 \in  \mathcal{B}( Y_{\epsilon}^0)$ we have \begin{align*}
(\Gamma_0 g_1 - \Gamma_0 g_2 )(s,t,x)=\int_{s}^{t}e^{\int_{s}^{\tau}a(\tau',x+\tau'-s)d\tau'} \int_{\mathbb{R}}^{}(g_1(\tau,t,y)-g_2(\tau,t,y))Q(x+\tau-s,dy)d\tau,
\end{align*}  from which we deduce \begin{align*}
     \left\| \Gamma_0 g_1 - \Gamma_0 g_2 \right\|_{\infty} &\leq \widehat{Q} \left\|  g_1 - g_2 \right\|_{\infty} \int_{s}^{t}e^{\int_{s}^{\tau}a(\tau',x+\tau'-s)d\tau'} d\tau \leq e^{\epsilon A}  (t-s)  \widehat{Q} \left\|  g_1 - g_2 \right\|_{\infty} \leq e^{\epsilon A} \epsilon  \widehat{Q} \left\|  g_1 - g_2 \right\|_{\infty}. 
\end{align*} So, for $ \epsilon e^{\epsilon A} < \frac{1}{ \widehat{Q}} $, $\Gamma_0$ is a contraction,  and we deduce from the Banach fixed point theorem the existence of a unique fixed point $\psi_0$ in $\mathcal{B}( Y_{\epsilon}^0)$. \\ Let us take $k \in 
[\![1;n]\!]$, we have $\psi_{k-1} \in \mathcal{B}( Y_{\epsilon}^{k-1})$ and we define the operator: $\Gamma_{k}:\mathcal{B}( Y_{\epsilon}^{k-1})  \rightarrow \mathcal{B}( Y_{\epsilon}^{k-1})  $ by  \begin{align*}
    \Gamma_{k} g(s,t,x) &=f_{k}(t,x+t-k\epsilon-s)e^{\int_{s}^{t-k\epsilon}a(\tau,x+\tau-s)d\tau}+ \int_{s}^{t-k\epsilon}e^{\int_{s}^{\tau}a(\tau',x+\tau'-s)d\tau'} \int_{\mathbb{R}}^{}g(\tau,t,y)Q(x+\tau-s,dy)d\tau 
\end{align*} where $f_k(t,\cdot)=\psi_{k-1}(t-k\epsilon,t,\cdot)$. For $g_1, g_2 \in  \mathcal{B}( Y_{\epsilon}^{k})$, we have \begin{align*}
(\Gamma_{k} g_1 - \Gamma_{k} g_2 )(s,t,x)=\int_{s}^{t-k\epsilon}e^{\int_{s}^{\tau}a(\tau',x+\tau'-s)d\tau'} \int_{\mathbb{R}}^{}(g_1(\tau,t,y)-g_2(\tau,t,y))Q(x+\tau-s,dy)d\tau,
\end{align*}  from which we deduce similarly as in the previous step that \begin{align*}
     \left\| \Gamma_k g_1 -  \Gamma_k  g_2 \right\|_{\infty} &\leq e^{\epsilon A}\epsilon  \widehat{Q} \left\|  g_1 - g_2 \right\|_{\infty} .
\end{align*} So, for $ \epsilon e^{\epsilon A} < \frac{1}{\widehat{Q}} $, $\Gamma_k$ is a contraction,  and we deduce from the Banach fixed point theorem the existence of a unique fixed point $\psi_k$ in $\mathcal{B}( Y_{\epsilon}^{k})$. As we have $Y^L=\cup^{\infty}_{k=1}Y_{\epsilon}^k$,  we define $\psi^L$ by $\psi_k$ in $Y_{\epsilon}^k$, and  we obtain  the existence of a unique fixed point $\psi$ in $\mathcal{B}_{\text{loc}}(Y)$.  If $f \in C_b(\mathbb{R})$, $\Gamma_k$ leaves $C_b(Y_\epsilon^k)$ invariant, and therefore the fixed point belongs to $ C(Y) \cap \mathcal{B}_{\text{loc}}(Y)$.   \\ 

Let $N \in \mathbb{N}$, we note  \begin{align*}
     Y_N&=\left\{  (s,t,x) \in \mathbb{R}^3, 0\leq s \leq t \hspace{0.2cm} x \in [-N,N]  \right\}  \\
     Y^k_{\epsilon,N}&=  \left\{  (s,t,x) \in \mathbb{R}^3, \hspace{0.2cm} \max(0,t-(k+1)\epsilon) \leq s \leq  t-k\epsilon, \hspace{0.2cm} x \in [-N,N]  \right\} .
\end{align*}  If $f \in C_b^1(\mathbb{R})$ we will do the fixed point in all $ \left\{  g \in  C_b^1(Y^k_{\epsilon,N}), \hspace{0.2cm} g(t-k\epsilon,t, \cdot)= \psi_{k-1}(t-k\epsilon,t,\cdot)  \right\}$  with the norm \begin{align*}
\left\| \cdot \right\|_{C^1}= \left\| \cdot \right\|_{\infty}+ \left\| \partial_x \cdot \right\|_{\infty} +\left\| \partial_s \cdot \right\|_{\infty} +\left\| \partial_t \cdot \right\|_{\infty}.
\end{align*} Let us denote for  $g \in C^1(Y) \cap \mathcal{B}_{\text{loc}}(Y)$ such that $g(t,t,\cdot)=f$ for all $t\geq 0$ 
\begin{align}
    I_1(s,t,x) &=f(x+t-s)e^{\int_{s}^{t}a(\tau,x+\tau-s)d\tau} \nonumber \\
    I_2(s,t,x) &= \int_{s}^{t}e^{\int_{s}^{\tau}a(\tau',x+\tau'-s)d\tau'} \int_{\mathbb{R}}^{}g(\tau,t,y)Q(x+\tau-s,dy)d\tau \nonumber \\
    &=\int_{0}^{t-s}e^{\int_{s}^{u+s}a(\tau',x+\tau'-s)d\tau'} \int_{\mathbb{R}}^{}g(u+s,t,y)Q(x+u,dy)du  \label{int1}\\
    &=\int_{s+x}^{t+x}e^{\int_{s}^{u-x}a(\tau',x+\tau'-s)d\tau'} \int_{\mathbb{R}}^{}g(u-x,t,y)Q(u-s,dy)du.  \label{int2}
\end{align} By differentiating $I_1$ we get  \begin{align*}
    \partial_t I_1(s,t,x)&= \left(f'(x+t-s) +  a(t,x+t-s)f(x+t-s)   \right) e^{\int_{s}^{t}a(\tau,x+\tau-s)d\tau} \\
    \partial_s I_1(s,t,x)&= \left(-f'(x+t-s) -  a(s,x) f(x+t-s) - f(x+t-s) \int_{s}^{t} \partial_x a(\tau,x+\tau-s) d\tau  \right) e^{\int_{s}^{t}a(\tau,x+\tau-s)d\tau} \\
    \partial_x I_1(s,t,x)&= \left(f'(x+t-s) +   f(x+t-s)  \int_{s}^{t} \partial_x a(\tau,x+\tau-s) d\tau   \right) e^{\int_{s}^{t}a(\tau,x+\tau-s)d\tau},
\end{align*}  and for $I_2$ using (\ref{int1}) we obtain \begin{align*}
    \partial_t  I_2(s,t,x) &=   e^{\int_{s}^{t}a(\tau,x+\tau-s)d\tau} \int_{\mathbb{R}}^{}f(y)Q(x+t-s,dy)+\int_{s}^{t}e^{\int_{s}^{\tau}a(\tau',x+\tau'-s)d\tau'} \int_{\mathbb{R}}^{}\partial_t g(\tau,t,y)Q(x+\tau-s,dy)d\tau\\
    \partial_s I_2(s,t,x)&= -e^{\int_{s}^{t}a(\tau,x+\tau-s)d\tau} \int_{\mathbb{R}}^{}f(y)Q(x+t-s,dy) \\
     & \qquad  +\int_{0}^{t-s} \left(a(u+s,x+u)-a(s,x) - \int_{s}^{u+s} \partial_x a(\tau',x+\tau'-s) d \tau' \right) e^{\int_{s}^{u+s} a(\tau',x+\tau'-s) d\tau'} \\
    &\hspace{0.3cm} \qquad \int_{\mathbb{R}}^{}g(u+s,t,y)Q(x+u,dy)du + \int_{0}^{t-s}e^{\int_{s}^{u+s}a(\tau',x+\tau'-s)d\tau'} \int_{\mathbb{R}}^{}\partial_s g(u+s,t,y)Q(x+u,dy)du \\
    &= -e^{\int_{s}^{t}a(\tau,x+\tau-s)d\tau} \int_{\mathbb{R}}^{}f(y)Q(x+t-s,dy) \\
    &  \qquad+\int_{s}^{t} \left(a(\tau,x+\tau-s)-a(s,x) - \int_{s}^{\tau} \partial_x a(\tau',x+\tau'-s) d \tau' \right) e^{\int_{s}^{\tau} a(\tau',x+\tau'-s) d\tau'} \\
    & \hspace{0.3cm} \qquad \int_{\mathbb{R}}^{}g(\tau,t,y)Q(x+\tau-s,dy)d\tau + \int_{s}^{t}e^{\int_{s}^{\tau}a(\tau',x+\tau'-s)d\tau'} \int_{\mathbb{R}}^{}\partial_s g(\tau,t,y)Q(x+\tau-s,dy)du. 
\end{align*} Using (\ref{int2}) we obtain 
\begin{align*}
       \partial_x I_2(s,t,x)&= e^{\int_{s}^{t}a(\tau,x+\tau-s)d\tau} \int_{\mathbb{R}}^{}f(y)Q(x+t-s,dy) - \int_{\mathbb{R}}^{} g(s,t,y) Q(x,dy) \\
    &  \qquad -\int_{s+x}^{t+x} \left(a(u-x,u-s) + \int_{s}^{u-x} \partial_x a(\tau',x+\tau'-s) d \tau' \right) e^{\int_{s}^{u-x} a(\tau',x+\tau'-s) d\tau'} \\
    &\hspace{0.3cm} \qquad \int_{\mathbb{R}}^{} g(u-x,t,y)Q(u-s,dy)du - \int_{s+x}^{t+x}e^{\int_{s}^{u-x}a(\tau',x+\tau'-s)d\tau'} \int_{\mathbb{R}}^{}\partial_s g(u-x,t,y)Q(u-s,dy)du \\
    &= e^{\int_{s}^{t}a(\tau,x+\tau-s)d\tau} \int_{\mathbb{R}}^{}f(y)Q(x+t-s,dy) - \int_{\mathbb{R}}^{} g(s,t,y) Q(x,dy) \\
    &  \qquad -\int_{s}^{t} \left(a(\tau,x+\tau-s) + \int_{s}^{\tau} \partial_x a(\tau',x+\tau'-s) d \tau' \right) e^{\int_{s}^{\tau} a(\tau',x+\tau'-s) d\tau'} \\
    & \hspace{0.3cm} \qquad\int_{\mathbb{R}}^{} g(\tau,t,y)Q(x+\tau-s,dy)d\tau - \int_{s}^{t}e^{\int_{s}^{\tau}a(\tau',x+\tau'-s)d\tau'} \int_{\mathbb{R}}^{}\partial_s g(\tau,t,y)Q(x+\tau-s,dy)d\tau \\
    &=e^{\int_{s}^{t}a(\tau,x+\tau-s)d\tau} \int_{\mathbb{R}}^{}f(y)Q(x+t-s,dy) - \int_{\mathbb{R}}^{} \left(  \int_{s}^{t} \partial_s g(s,t,y) ds  - f(y) \right) Q(x,dy) \\
    &  \qquad-\int_{s}^{t} \left(a(\tau,x+\tau-s) + \int_{s}^{\tau} \partial_x a(\tau',x+\tau'-s) d \tau' \right) e^{\int_{s}^{\tau} a(\tau',x+\tau'-s) d\tau'} \\
    & \hspace{0.3cm}\qquad \int_{\mathbb{R}}^{} g(\tau,t,y)Q(x+\tau-s,dy)d\tau - \int_{s}^{t}e^{\int_{s}^{\tau}a(\tau',x+\tau'-s)d\tau'} \int_{\mathbb{R}}^{}\partial_s g(\tau,t,y)Q(x+\tau-s,dy)d\tau. 
\end{align*} So we deduce that for all $g_1, g_2 \in  \left\{  g \in  C_b^1 (Y^0_{\epsilon,N}), \hspace{0.2cm}  g(t,t, \cdot)= f \right\}$ \begin{align*}
    \left\| \partial_t \Gamma_0 g_1 - \partial_t \Gamma_0 g_2 \right\|_{\infty} &\leq  e^{\epsilon A} \epsilon\widehat{Q}  \left\| \partial_t g_1 - \partial_t g_2 \right\|_{\infty}\\
     \left\| \partial_s \Gamma_0 g_1 - \partial_s \Gamma_0 g_2 \right\|_{ \infty } &\leq  3Ae^{\epsilon A} \epsilon\widehat{Q}  \left\|  g_1 -  g_2 \right\|_{\infty}+e^{\epsilon A} \epsilon \widehat{Q}  \left\| \partial_s g_1 - \partial_s g_2 \right\|_{ \infty } \\
     \left\| \partial_x \Gamma_0 g_1 - \partial_x \Gamma_0 g_2 \right\|_{ \infty } &\leq  \epsilon \widehat{Q}\left\|\partial_s  g_1 -  \partial_s g_2 \right\|_{\infty}+Ae^{\epsilon A}\epsilon \widehat{Q} \left\|  g_1 -  g_2 \right\|_{\infty}+ \left(\underset{[0,T] \times [-N,N+T]}{\sup} \partial_x a \right) \epsilon^2 e^{\epsilon A} \widehat{Q} \left\|  g_1 -  g_2 \right\|_{\infty} \\ 
     &\qquad +e^{\epsilon A} \epsilon \widehat{Q}  \left\| \partial_s g_1 - \partial_s g_2 \right\|_{ \infty } \\
     & \leq \widehat{Q} \left(  Ae^{\epsilon A}\epsilon+\left(\underset{[0,T] \times [-N,N+T]}{\sup} \partial_x a \right)\epsilon^2 e^{\epsilon A}\right)\left\|  g_1 -  g_2 \right\|_{\infty}+ \left( \widehat{Q}\epsilon + e^{\epsilon A} \epsilon \widehat{Q} \right) \left\| \partial_s g_1 - \partial_s g_2 \right\|_{ \infty }.
\end{align*} So, we obtain \begin{align*}
     \left\| \Gamma_0 g_1 - \Gamma_0 g_2 \right\|_{C^1} &\leq \widehat{Q} \left(  3 e^{\epsilon A} \epsilon+ 4Ae^{\epsilon A} \epsilon+ \left(\underset{[0,T] \times [-N,N+T]}{\sup} \partial_x a \right)e^2e^{\epsilon A} + \epsilon \right) \left\|  g_1 -  g_2 \right\|_{C^1}.
\end{align*} As $\left(  3 e^{\epsilon A} \epsilon+ 4Ae^{\epsilon A} \epsilon+ \left(\underset{[0,T] \times [-N,N+T]}{\sup} \partial_x a \right)e^2e^{\epsilon A} + \epsilon \right)< \frac{1}{ \widehat{Q}} $,  $\Gamma_0$ is a contraction,  and we deduce from the Banach fixed point theorem the existence of a unique fixed point $\psi_0$ in $C_b^1(Y^0_{\epsilon,N})$. In the same way that we did for continuous functions, we obtain by induction a unique fixed point  $\psi_k$ in $C_b^1(Y^k_{\epsilon,N})$.
As we have $Y_N=\cup^{n}_{k=1}Y^k_{\epsilon,N}$, we define $\psi_N$ by $\psi_N=\psi_k$ in $Y^k_{\epsilon,N}$. We then obtain  the existence of a unique fixed point $\psi_N$ in $C_b^1( \cup^{n}_{k=1}\overset{\circ}{\widehat{Y^k_{\epsilon,N}}})$. \\ To show that the fixed point is differentiable at the boundary of the sets we denote for all  $k \in [\![1;2n-1]\!]$

\begin{align*}
    \widetilde{Y}^0_{\epsilon,N}&= \left\{  (s,t,x) \in \mathbb{R}^3, \hspace{0.2cm} \max(0,t-\epsilon/2) \leq s \leq  t, \hspace{0.2cm} -N \leq x \leq N  \right\} \\
    \widetilde{Y}^k_{\epsilon,N}&= \left\{  (s,t,x) \in \mathbb{R}^3, \hspace{0.2cm} \max(0,t-(2k+3)\epsilon/2) \leq s \leq  t-(2k+1)\epsilon/2, \hspace{0.2cm} -N \leq x \leq N   \right\}, 
\end{align*}  and define the operators $\widetilde\Gamma_k: C_b( \widetilde{Y}^k_{\epsilon,N})  \rightarrow C_b(\widetilde{Y}^k_{\epsilon,N})  $. In the same way we prove that $\widetilde\Gamma_k$  is a contraction for all $k \in [\![0;n]\!]$ , so again from the Banach fixed point theorem we obtain a fixed point $\widetilde{\psi}_N$ in  $C_b^1( \cup_{k=1}^{2n-1} \widetilde{Y}^k_{\epsilon,N})$. By uniqueness of the fixed point we have $\psi_N=\widetilde{\psi}_N$ so we deduce that $\psi_N  \in C_b^1(Y_N)$.  We define $\psi$ by $\psi(x)=\psi_N(x)$ if $x \in [-N,N]$, by uniqueness of the fixed point we clearly have $\psi_{N} \Big|_{(-M,M)} = \psi_M $ for $M \leq N$, so $\psi$ is well defined, and we obtain that $\psi \in C^1_b(Y) $.  \\ \\ To conclude the proof let us show the properties \begin{align*}
\partial_tM_{s,t}f=M_{s,t}\mathcal{L}_tf, \hspace{0.2cm} \partial_sM_{s,t}f=-\mathcal{L}_sM_{s,t}f.
\end{align*}  From the computation of $\partial_t \Gamma_0 $, we have \begin{align*}
    \partial_t \Gamma_0  \psi(s,t,x)= \mathcal{L}_t f(x+t-s)e^{\int_{s}^{t}a(\tau,x+\tau-s)d\tau} +  \int_{s}^{t}e^{\int_{s}^{\tau}a(\tau',x+\tau'-s)d\tau'} \int_{\mathbb{R}}^{}\partial_t \psi(\tau,t,y)Q(x+\tau-s,dy)d\tau.
\end{align*} So, if $\Gamma_0  \psi=\psi$, we obtain \begin{align*}
    \partial_t  \psi(s,t,x)= \mathcal{L}_t f(x+t-s)e^{\int_{s}^{t}a(\tau,x+\tau-s)d\tau} +  \int_{s}^{t}e^{\int_{s}^{\tau}a(\tau',x+\tau'-s)d\tau'} \int_{\mathbb{R}}^{}\partial_t \psi(\tau,t,y)Q(x+\tau-s,dy)d\tau,
\end{align*} so, if $\psi$ is the fixed point of $\Gamma_0 $ with terminal condition $f$, then $\partial_t \psi$ is the fixed point of $ \Gamma_0 $ with terminal condition $\mathcal{L}_tf$. By uniqueness we deduce that $\partial_tM_{s,t}\psi=M_{s,t}\mathcal{L}_t \psi $ and from the computation of $\partial_s \Gamma_0 $ and $\partial_x \Gamma_0  $  we have \begin{align*}
\partial_s \Gamma_0  \psi(s,t,x)&= -\partial_x \Gamma_0  \psi(s,t,x) - a(s,x) f(x+t-s)e^{\int_{s}^{t}a(\tau,x+\tau-s)d\tau} -\int_{\mathbb{R}}^{} \psi(s,t,y) Q(x,dy) \\
& \qquad  -\int_{s}^{t}a(s,x)e^{\int_{s}^{\tau}a(\tau',x+\tau'-s)d\tau'}\int_{\mathbb{R}}^{}\psi(\tau,t,y)Q(x+\tau-s,dy)d\tau \\
&= -\partial_x \Gamma_0  \psi(s,t,x) -\int_{\mathbb{R}}^{} \psi(s,t,y) Q(x,dy) \\
& \qquad- a(s,x) \left( f(x+t-s)e^{\int_{s}^{t}a(\tau,x+\tau-s)d\tau}  -\int_{s}^{t}e^{\int_{s}^{\tau}a(\tau',x+\tau'-s)d\tau'}\int_{\mathbb{R}}^{}\psi(\tau,t,y)Q(x+\tau-s,dy)d\tau \right) \\
&=-\partial_x \Gamma_0  \psi(s,t,x) -a(s,x) \Gamma_0  \psi(s,t,x)-\int_{\mathbb{R}}^{} \psi(s,t,y) Q(x,dy). 
\end{align*} If $\psi$ is the fixed point of $\Gamma_0$  with terminal condition $f$, we then obtain $  \partial_s\psi(s,t,x)= -\mathcal{L}_s \psi $ from which we deduce that $ \partial_s M_{s,t} \psi = -\mathcal{L}_s M_{s,t}\psi.  $

\end{proof} \subsection{Construction of measure solutions} We start by the definition of a  measure solution for our PDE. \begin{deff}\label{defmesuresol}
A family $(\mu_{s,t})$ is called a measure solution to Equation (\ref{yop}) if for all bounded and continuously differentiable function $f $ the mapping $(s,t) \rightarrow \langle \mu_{s,t}, f \rangle   $ is continuous and for all $t \geq s \geq 0 $ \begin{align*}
 \langle \mu_{s,t}, f \rangle= \langle \mu_{s,s} , f \rangle+\int_{s}^{t} \langle \mu_{s,\tau} , \mathcal{L}_\tau f  \rangle d \tau.
\end{align*}
\end{deff} For $\mu \in \mathcal{M}_+(  X )$ and $t \geq s \geq 0$ and a borel set $A$ we define \begin{align*}
\left( \mu M_{s,t} \right)(A)= \int_{\mathbb{R}}^{} M_{s,t} \mathbb{1}_{A} d\mu . 
\end{align*} We will start by proving that $\mu M_{s,t}$ defines a positive measure and verifies the duality relation.

\begin{monTheoreme}
For all $\mu \in \mathcal{M}_+(X) $ and all $t \geq s \geq 0$, $\mu M_{s,t}$ defines a positive measure. Additionally for all bounded measurable function $f$ we have the relation

$$ \langle \mu M_{s,t} ,f  \rangle = \langle \mu , M_{s,t} f \rangle.$$

\end{monTheoreme}

\begin{proof}
Clearly, as $M_{s,t}$ is a positive operator we have that $\mu M_{s,t}$ is positive and $(\mu M_{s,t})(\emptyset)=0$. Let $(A_n)_{n \geq 0}$ be a countable sequence of disjoint Borel sets of $ X  $ and define $f_n=\sum_{k=0}^n \mathbb{1}_{A_k}=\mathbb{1}_{\cup_{k=0}^n A_k}$. For any integer $n$ we have  \begin{align*}
    \mu M_{s,t} \left(  \cup_{k=0}^n A_k  \right) &= \int_{\mathbb{R}}^{}M_{s,t} f_n d\mu =\sum_{k=0}^n \int_{\mathbb{R}}^{} M_{s,t} \mathbb{1}_{A_k}d \mu 
    = \sum_{k=0}^n (\mu M_{s,t}) (A_k) .
\end{align*} We apply the monotone convergence theorem to the relation (\ref{2.7}) and we obtain by uniqueness of the fixed point that $ \lim\limits_{n \rightarrow +\infty} \left( M_{s,t} f_n(x) \right)= M_{s,t} \left( \lim\limits_{n \rightarrow +\infty} f_n(x) \right)$ for all $t \geq s \geq 0 $  and  $x \in X $. Moreover, $(f_n)_{n \geq 0}$ is increasing and it's pointwise limit is $f= \mathbb{1}_{\cup_{k=0}^{\infty} A_k }$, so we deduce by monotone convergence \begin{align*}
     \lim\limits_{n \rightarrow +\infty} \int_{\mathbb{R}}^{} M_{s,t} f_n d\mu 
    = \int_{\mathbb{R}}^{} M_{s,t} f d\mu 
    = \mu M_{s,t} \left( \cup_{k=0}^{\infty} A_k \right).
\end{align*} So we conclude that $ \mu M_{s,t} \left( \cup_{k=0}^{\infty} A_k \right)= \sum_{k=0}^{\infty} (\mu M_{s,t}) (A_k)  $ and $\mu M_{s,t}$ satisfies the definition of a positive measure. For a simple function $f$ we have clearly the identity $\langle \mu M_{s,t} ,f \rangle = \langle \mu , M_{s,t} f \rangle$. We deduce by using $\lim\limits_{n \rightarrow +\infty} \left( M_{s,t} f_n(x) \right))= M_{s,t} \left( \lim\limits_{x \rightarrow +\infty} f_n(x) \right)$ that the relation $\langle \mu M_{s,t} ,f \rangle = \langle \mu , M_{s,t} f \rangle$ is also true for any non negative measurable bounded function. Finally, decomposing $f=f_+-f_-$ we obtain that for all function $f$ measurable $ \langle \mu M_{s,t} ,f \rangle = \langle \mu , M_{s,t} f \rangle.$

\end{proof} We can now finish this part by the construction of a measure solution of (\ref{yop}).

\begin{lemma}\label{hilbert}
  If the family $( \mu_{s,t})_{t \geq s \geq 0} \subset \mathcal{M}(\mathbb{R})$  is  solution to Equation (\ref{yop}), in the sense of Definition \ref{defmesuresol}, then  the mapping $(s,t) \rightarrow \mu_{s,t}$ is continuous and for all $\psi \in C^1_b((s,\infty) \times\mathbb{R})$ with compact support in time,
  
  \begin{align*}
         \int_{s}^{\infty}  \int_{\mathbb{R}}^{} \left( \partial_t \psi(t,x) +\mathcal{L}_t \psi(t,x) \right) d\mu_{s,t}(x) dt + \int_{\mathbb{R}}^{} \psi(s,x) d\mu_{s,s}(x) =0 . 
  \end{align*}

\end{lemma}

\begin{proof}

Assume that $(\mu_{s,t})_{t \geq s \geq 0}$ satisfies Definition \ref{defmesuresol} and let $\psi \in C^1_b((s,\infty) \times\mathbb{R})$ with compact support in time, we use $f=\partial_t \psi(t, \cdot)$ as a test function in Definition \ref{defmesuresol}

\begin{align*}
    &\int_{s}^{\infty}  \int_{\mathbb{R}}^{} \partial_t \psi(t,x)  d\mu_{s,t}(x) dt\\
    &=  \int_{s}^{\infty}  \int_{\mathbb{R}}^{} \partial_t \psi(t,x)  d\mu_{s,s}(x) dt +  \int_{s}^{\infty}  \int_{s}^{t}  \int_{\mathbb{R}}^{}  \mathcal{L}_\tau(\partial_t \psi) d\mu_{s,\tau}(x)d\tau dt \\
    &=\int_{s}^{\infty} \int_{\mathbb{R}}^{} \partial_t \psi(t,x)  d\mu_{s,s}(x) dt +  \int_{s}^{\infty}  \int_{s}^{t}  \int_{\mathbb{R}}^{} \left( \partial_x \partial_t \psi(t,x)+a(\tau,x)\partial_t\psi(t,x)+ \left(\int_{\mathbb{R}}^{} \partial_t \psi(t,y)Q(x,dy) \right) \right) d\mu_{s,\tau}(x) d\tau dt \\
    &= \int_{\mathbb{R}}^{}  \left(\int_{s}^{\infty} \partial_t \psi(t,x)   dt \right) d\mu_{s,s}(x) +  \int_{\mathbb{s}}^{\infty}  \int_{\mathbb{R}}^{} \left( \int_{\tau}^{\infty}  \left[\partial_t \left( \partial_x \psi(t,x)+a(\tau,x) \psi(t,x) \right)+ \partial_t \left(\int_{\mathbb{R}}^{} \psi(t,y)Q(x,dy) \right)\right] dt \right)  d\mu_{s,\tau}(x) d\tau  \\
    &= -\int_{\mathbb{R}}^{}    \psi(s,x)    d\mu_{s,s}(x) -  \int_{s}^{\infty}  \int_{\mathbb{R}}^{} \left( \partial_x \psi(\tau,x)+a(\tau,x)\psi(\tau,x)+  \left(\int_{\mathbb{R}}^{} \psi(\tau,y)Q(x,dy) \right) \right)  d\mu_{s,\tau}(x) d\tau .
\end{align*}

\end{proof}

\begin{monTheoreme} 
  For any measure $\mu \in \mathcal{M}(\mathbb{R})$ the family $(\mu M_{s,t})_{t \geq s \geq 0}$ is the unique solution to (\ref{yop}),  in the sense of Definition \ref{defmesuresol}, with initial distribution $\mu$.
\end{monTheoreme}

\begin{proof}

We start by checking that for all $s \in \mathbb{R}$, $t \in (s,\infty)  \longmapsto   \langle \mu M_{s,t} , f \rangle$ is continuous for all bounded and continuously differentiable function $f$. Due to the linearity, it's sufficient to check that $\lim\limits_{t \rightarrow s} \langle \mu,  M_{s,t}f \rangle=\langle \mu, f \rangle $. We write

\begin{align*}
      \left|  \langle \mu,  M_{s,t}f \rangle -  \langle  \mu, f \rangle \right| & \leq   \left|  \langle \mu, f \rangle - \int_{\mathbb{R}}^{} f(x+t-s)e^{ \int_{s}^{t}a(\tau,x+\tau-s)d\tau}d \mu \right| +  \\ 
      & \qquad \left| \langle \mu , \int_{s}^{t}{ e^{\int_{s}^{\tau}{a(\tau',x+\tau'-s)d\tau'}} \int_{\mathbb{R}}^{}{M_{\tau,t}f(y)Q(x+ \tau-s,dy)}d\tau} \rangle  \right| \\
      & \leq  \left|  \langle \mu, f \rangle - \int_{\mathbb{R}}^{} f(x+t-s)e^{ \int_{s}^{t}a(\tau,x+\tau-s)d\tau}d \mu \right| + \left(\underset{\tau \in (s,t)}{\sup} M_{\tau,t}f \right)\widehat{Q} \| \mu  \|_{TV}  e^{A(t-s}(t-s).
\end{align*} The first part goes to $0$ by dominated convergence and also the second term because the semiflow is bounded for any function $f$ bounded , which shows that  $t \in (s,\infty) \rightarrow  \langle \mu M_{s,t} , f \rangle$ is continuous. Let us take $f \in C^1_b(\mathbb{R}) $, from $\partial_t M_{s,t}f= M_{s,t} \mathcal{L}_t f$, we deduce by integration that for all $x \in \mathbb{R}$, we have that

$$M_{s,t} f(x) =f(x)+ \int_{s}^{t} M_{s,t}\mathcal{L}_tf(x) dt.$$ We deduce that, by Fubini's theorem,

 $$\langle \mu ,M_{s,t} f \rangle =\langle \mu ,f \rangle + \int_{s}^{t} \langle \mu ,M_{s,t}(\mathcal{L}_tf)\rangle dt. $$ So, we deduce from the identity $\langle \mu M_{s,t} ,f \rangle = \langle \mu , M_{s,t} f \rangle$ that $(\mu M_{s,t})$ satisfies the existence part in Definition $8$. \\  Let us now check the uniqueness. Let $s \in \mathbb{R} $ and $\mu_{s,t}^1, \mu_{s,t}^2$ be two solutions to Equation~\eqref{yop} with $\mu_{s,s}^1=\mu_{s,s}^2=\mu$. Let us set $ \mu_{s,t}=\mu_{s,t}^1-\mu_{s,t}^2 $.
 In order to prove the uniqueness we need to check that $\mu_{s,t}=0$ for all $t \geq s$.
 From Lemma \ref{hilbert} we have
 $$   \int_{s}^{+\infty} \! \int_{\mathbb{R}}^{} \left(\partial_t \psi(t,x) +\mathcal{L}_t \psi(t,x) \right) d\mu_{s,t}(x) dt=0, $$ for all $\psi \in C^1_b((s,\infty) \times\mathbb{R})$ with compact support in time.
 Let us take $\phi \in  C^1_c((s,\infty) \times\mathbb{R})$ and let $T>s$ such that $\text{supp}  (\phi) \subset [s,T) \times \mathbb{R} $.
 Using the same method as in Theorem $7$, we can prove the existence of a solution $\psi \in C^1_b([s,T) \times \mathbb{R})$ such that for all $(t,x) \in [s,T) \times \mathbb{R}$ $$ \partial_t \psi(t,x) +\mathcal{L}_t \psi(t,x) =\phi(t,x), $$ with terminal condition $ \psi(T,x)=0$.
 Since $\phi \in C^1_c([s,T) \times \mathbb{R}) $, we easily check that the extension of $\psi(t,x)$ by~$0$ for $t>T$ belongs to $C^1_b((s,\infty) \times\mathbb{R})$, is compactly supported in time and still satisfies  $ \partial_t \psi(t,x) +\mathcal{L}_t \psi(t,x) =\phi(t,x)$.
 So we obtain
 $$  \int_{s}^{+\infty} \! \int_{\mathbb{R}}^{} \phi(t,x) \,d\mu_{s,t}(x)dt=0 $$
 for all $\phi \in C^1_c((s,\infty) \times\mathbb{R})$.
 This ensures that $\mu_{s,t}=0$, or in other words $\mu^1_{s,t}=\mu^2_{s,t}$,  for all $t \geq s$. 
\end{proof}

\subsection{ Proof of Theorem \ref{ntm}} Before proving the main
theorem of this section, Theorem \ref{ntm}, we start    with a useful strong positivity result about  $M_{s,t}$.  \begin{lemma} \label{lemposiv}
 Let $t>s>0$ and $x_1,x_2,y_1,y_2 \in \mathbb{R}$ with $x_1<x_2$ and $y_1<y_2$. Then there exists $\eta>0$ such that  \begin{align*}
M_{s,s+t }\mathbb{1}_{[x_1,x_2]} \geq \eta \mathbb{1}_{[y_1,y_2]}.
\end{align*} 
\end{lemma} \begin{proof} We use the Duhamel formula (\ref{2.5}) with $f=\mathbb{1}_{[x_1,x_2]}$, and for all $ s,t \geq 0$ \begin{align*}
M_{s,s+t} \mathbb{1}_{[x_1,x_2]}(x)  & \geq  \int_{s}^{t+s} e^{\int_{s}^{\tau} a(\tau',x+\tau'-s)d\tau'  }\int_{\mathbb{R}}^{}  M_{\tau,s+t}\mathbb{1}_{[x_1,x_2]}(y) Q(x+\tau-s,dy) d\tau \\
& \geq  \int_{s}^{t+s} \left(e^{\int_{s}^{\tau} a(\tau',x+\tau'-s)d\tau'  } \right)\left(e^{\int_{\tau}^{s+t} a(\tau',x+\tau'-\tau)d\tau'  } \right) \int_{\mathbb{R}}^{}  \mathbb{1}_{[x_1,x_2]}(y+t+s-\tau) Q(x+\tau-s,dy) d\tau \\
& \geq  \int_{s}^{t+s} e^{\int_{s}^{\tau} \underline{a}(x+\tau'-s)d\tau'  } e^{\int_{\tau}^{s+t} \underline{a}(x+\tau'-\tau)d\tau'  }  \int_{\mathbb{R}}^{}  \mathbb{1}_{[x_1,x_2]}(y+t+s-\tau) Q(x+\tau-s,dy) d\tau \\
& \geq  e^{t \inf_{(x,x+t)} \underline{a}} \int_{s}^{t+s} \int_{\mathbb{R}}^{} \mathbb{1}_{[x_1,x_2]}(y+t+s-\tau) Q(x+\tau-s,dy) d\tau  \\
& \geq \kappa_0 e^{ t \inf_{(x,x+t)} \underline{a}} \int_{s}^{t+s} \int_{\mathbb{R}}^{}  \mathbb{1}_{[x_1,x_2]}(y+t+s-\tau)  \mathbb{1}_{[x+\tau-s-\epsilon,x+\tau-s+\epsilon]}(y) dy d\tau  \\
& \geq \kappa_0 e^{ t \inf_{(x,x+t)} \underline{a}}\int_{s}^{s+t}  \int_{x+\tau-s-\epsilon}^{x+\tau-s+\epsilon}  \mathbb{1}_{[x_1,x_2]}(y+t+s-\tau)   dy d\tau \\
& \geq \kappa_0 e^{ t \inf_{(x,x+t)} \underline{a}}\int_{s}^{s+t}  \int_{x+t-\epsilon}^{x+t+\epsilon}  \mathbb{1}_{[x_1,x_2]}(u) du d\tau \\
& \geq \kappa_0 t e^{ t \inf_{(x,x+t)} \underline{a}}  \int_{x+t-\epsilon}^{x+t+\epsilon} \mathbb{1}_{[x_1,x_2]}(u) du.
\end{align*} So we have   $M_{s,s+t} \mathbb{1}_{[x_1,x_2]}(x) \geq \text{const.} >0$ for all $x\in[x_1-t-\epsilon/2,x_2-t+\epsilon/2]$. Let $n \in \mathbb{N}$, $\tau >0$ and $x\in[x_1-\tau/n-\epsilon/2,x_2-\tau/n+\epsilon/2]$ ,  considering $t=\tau/n$, we get that $M_{s,s+t} \mathbb{1}_{[x_1,x_2]}(x) \geq c_0$ for some $c_0>0$. In the same way we have $ \ M_{s+(k-1)t,s+k t} \mathbb{1}_{[x_1,x_2]}(x) \geq c_0 >0.$ Let us show by induction that \begin{align*}
  M_{s,s+n t} \mathbb{1}_{[x_1,x_2]}(x) \geq c_0^n \mathbb{1}_{[x_1-\tau-n\epsilon/2,x_2-\tau+n\epsilon/2]},
\end{align*} and we will conclude by taking $n$ large enough. We already show it for $n=1$. For $n-1 \rightarrow n$,  \begin{align*}
    M_{s,s+nt} \mathbb{1}_{[x_1,x_2]}(x) &=   M_{s,s+(n-1)t}  M_{s+(n-1)t, s+nt } \mathbb{1}_{[x_1,x_2]}(x) \\
    & \geq c_0^{n-1}  M_{s+(n-1)t, s+nt }   \mathbb{1}_{[x_1-\tau-(n-1)\epsilon/2,x_2-\tau+(n-1)\epsilon/2]} \\
    & \geq c_0^n \mathbb{1}_{[x_1-\tau-n\epsilon/2,x_2-\tau+n\epsilon/2]}.
\end{align*}

\end{proof} Now, we have all the tools to demonstrate the main theorem, Theorem~\ref{ntm}.   We will verify assumptions B with the function $V=1$ and the set $K=(-R,R)$ with $R>0$ large enough to be chosen in the proof.  \begin{proof}[Proof of Theorem~\ref{ntm}]

\textbf{Assumption (B3)} For any $t_0,s  \geq 0$ and any $x \in \mathbb{R}$, we have 
\begin{align*}
 M_{s+t_0/2,s+t_0} f(x)  & \geq  \int_{s+t_0/2}^{s+t_0} e^{\int_{s+t_0/2}^{\tau} a(\tau',x+\tau'-s-t_0/2)d\tau'  }\int_{\mathbb{R}}^{}  M_{\tau,s+t_0}f(y) Q(x+\tau-s,dy) d\tau \\
& \hspace{-20mm} \geq  \int_{s+t_0/2}^{s+t_0} e^{\int_{s+t_0/2}^{\tau} a(\tau',x+\tau'-s-t_0/2)d\tau'  }  \int_{\mathbb{R}} e^{\int_{\tau}^{s+t_0} a(\tau',x+\tau'-\tau)d\tau'  } f(y+s+t_0-\tau) Q(x+\tau-s-t_0/2,dy) d\tau \\
& \hspace{-10mm} \geq  \int_{s+t_0/2}^{s+t_0} e^{\int_{s+t_0/2}^{\tau} \underline{a}(x+\tau'-s-t_0/2)d\tau'  } e^{\int_{\tau}^{s+t_0} \underline{a}(x+\tau'-\tau)d\tau'  }  \int_{\mathbb{R}}^{} f(y+s+t_0-\tau) Q(x+\tau-s-t_0/2,dy) d\tau \\
& \geq  e^{\frac{t_0}{2} \inf_{(x,x+t_0)}  \underline{a}} \int_{s+t_0/2}^{s+t_0} \int_{\mathbb{R}}^{}  f(y+s+t_0-\tau) Q(x+\tau-s-t_0/2,dy) d\tau  \\
& \geq \kappa_0 e^{\frac{t_0}{2} \inf_{(x,x+t_0)}  \underline{a}} \int_{s+t_0/2}^{s+t_0} \int_{\mathbb{R}}^{}  f(y+s+t_0-\tau)  \mathbb{1}_{(x+\tau-s-t_0/2-\epsilon,x+\tau-s-t_0/2+\epsilon)}(y) dy d\tau  \\
& \geq \kappa_0 e^{\frac{t_0}{2} \inf_{(x,x+t_0)}  \underline{a}}\int_{s+t_0/2}^{s+t_0}  \int_{x+\tau-s-t_0/2-\epsilon}^{x+\tau-s-t_0/2+\epsilon} f(y+s+t_0-\tau)  dy d\tau \\
& \geq \kappa_0 e^{\frac{t_0}{2} \inf_{(x,x+t_0)}  \underline{a}}\int_{s+t_0/2}^{s+t_0}  \int_{x+t_0/2-\epsilon}^{x+t_0/2+\epsilon} f(u) du d\tau \\
& \geq \frac{\kappa_0 t_0}{2} e^{\frac{t_0}{2} \inf_{(x,x+t_0)}  \underline{a}}\int_{x+t_0/2-\epsilon}^{x+t_0/2+\epsilon} f(u) du.
\end{align*} 
We define $\nu=\epsilon^{-1}\mathbb{1}_{\left(\frac{t_0-\epsilon}{2},\frac{t_0+\epsilon}{2}\right)}\lambda$, where $\lambda$ is the Lebesgue measure. From this, we deduce that for all $t_0,s \geq 0$, there exist $c_0>0$ such that, $$M_{s+t_0/2,s+t_0}f(x) \geq c_0 \langle \nu, f \rangle \mathbb{1}_{(-\epsilon/2, \epsilon/2)}.$$
Applying $M_{s,s+t_0/2}$ to this inequality and using Lemma 12, we infer that for all $R>0$ there is $c>0$ such that $$M_{s,s+t_0}f \geq c\langle \nu, f \rangle \mathbb{1}_{(-R,R)} $$ for all $f\geq0$.  \\

\textbf{Assumption (B5)}  We will build by induction two families $(\sigma_{x,y}^{t-s,n})$ et $(c_{x,y}^{t-s,n})$ indexed by $n\in \mathbb{N}$, $t\geq s$, $x,y \in \mathbb{R}$ such that $\sigma_{x,y}^{t-s,n}$ is a probability measure on $[0,t]$ which has a positive Lebesgue density $\mathfrak{s}_{x,y}^{t-s,n}$, $c_{x,y}^{t-s,n}$ is a positive constant and $(x,y,u) \mapsto c_{x,y}^{t-s,n}$ is continuous and for all $f \geq 0$ \begin{align} \label{recuu}
    M_{s,t}f(x) \geq c_{x,y}^{t-s,n} \int_{s}^{t}{M_{u,t}f(y)}\sigma_{x,y}^{t,n}(du). 
\end{align} \underline{For $n=0$:} For
$y=x+t-s $, Duhamel formula (\ref{2.5}) ensures that for any $f\geq 0$ \begin{align*}
 M_{s,t} f(x) \geq f(y)e^{\int_{s}^{t}a(\tau,x+\tau-s)d\tau}. 
\end{align*} It gives $(\ref{recuu})$ with $\sigma_{x,y}^{t-s,0}=\delta_{t}$ and $c_{x,y}^{t-s,0}=e^{\int_{s}^{t}a(\tau,x+\tau-s)d\tau}$. \\ \\ \underline{For $n \rightarrow n+1$:}  The induction hypothesis ensures that \begin{align*}
M_{\tau,t}f(z) \geq c_{z,y}^{t-\tau,n} \int_{\tau}^{t}M_{u,t}f(y) \mathfrak{s}_{z,y}^{t-\tau,n}(u)du . 
\end{align*} So by the Duhamel formula (\ref{2.5}), we obtain \begin{align*}
    M_{s,t}f(x) &\geq \kappa_0 \int_{s}^{t}{ e^{\int_{s}^{\tau}{a(\tau',x+\tau'-s)d\tau'}}  \int_{x+\tau-s-\epsilon}^{x+\tau-s+\epsilon}  c_{z,y}^{t-\tau,n} \int_{\tau}^{t}M_{u,t}f(y)  \mathfrak{s}_{z,y}^{t-\tau,n}(u)du    dz    d\tau} \\
    &=\kappa_0 \int_{s}^{t}{ e^{\int_{s}^{\tau}{a(\tau',x+\tau'-s)d\tau'}}  \int_{\tau}^{t}  \int_{x+\tau-s-\epsilon}^{x+\tau-s+\epsilon}   c_{z,y}^{t-\tau,n} M_{u,t}f(y)  \mathfrak{s}_{z,y}^{t-\tau,n}(u)dz    du    d\tau} \\
    &=\kappa_0 \int_{s}^{t}{   \int_{s}^{u}e^{\int_{s}^{\tau}{a(\tau',x+\tau'-s)d\tau'}}  \int_{x+\tau-s-\epsilon}^{x+\tau-s+\epsilon}   c_{z,y}^{t-\tau,n} M_{u,t}f(y)  \mathfrak{s}_{z,y}^{t-\tau,n}(u)dz    d\tau   du} \\
    &= \kappa_0 \int_{s}^{t}{   M_{u,t}f(y) \int_{s}^{u}e^{\int_{s}^{\tau}{a(\tau',x+\tau'-s)d\tau'}}  \int_{x+\tau-s-\epsilon}^{x+\tau-s+\epsilon}   c_{z,y}^{t-\tau,n}  \mathfrak{s}_{z,y}^{t-\tau,n}(u)dz    d\tau   du}, 
\end{align*} which gives (\ref{recuu}) with \begin{align*}
c_{x,y}^{t-s,n+1}=\kappa_0 \int_{s}^{t}{  \int_{s}^{u}e^{\int_{s}^{\tau}{a(\tau',x+\tau'-s)d\tau'}} \int_{x+\tau-s-\epsilon}^{x+\tau-s+\epsilon}   c_{z,y}^{t-\tau,n}  \mathfrak{s}_{z,y}^{t-\tau,n}(u)dz    d\tau   du}, 
\end{align*} and $\sigma_{x,y}^{t-s,n+1}(ds)=\mathfrak{s}_{x,y}^{t-s,n+1}(s)ds$, with \begin{align*}
\mathfrak{s}_{x,y}^{t-s,n+1}=\frac{\kappa_0}{c_{x,y}^{t,n+1}}\int_{s}^{u}e^{\int_{s}^{\tau}{a(\tau',x+\tau'-s)d\tau'}}  \int_{x+\tau-s-\epsilon}^{x+\tau-s+\epsilon}   c_{z,y}^{t-\tau,n}  \mathfrak{s}_{z,y}^{t-\tau,n}(u)dz    d\tau   du. 
\end{align*} \\ \textbf{Assumption (B2)} Let us take $t \geq s \geq 0$,  and consider $x_0\geq\epsilon$ to be chosen later (in the proof of Assumption~(B1), and consider the function \begin{align*}
 \psi_0(x)=\left(1-(\frac{x}{x_0})^2 \right)^2\mathbb{1}_{[-x_0,x_0]}(x).
\end{align*}  Let us consider the generator $ \overline{\mathcal{L}}$ defined by  $\overline{\mathcal{L}}f(x)=f'(x)+\underline{a}(x)f(x)+\int_{\mathbb{R}}^{}f(y)Q(x,dy) $ for all $f \in C^1_b(\mathbb{R}) $, and  $ (\overline{M}_{s,t})_{t \geq s \geq 0 }  $  the associated  semiflow. We note that, since the operator $\mathcal L$ does not depend on $t$, we can also associate a semigroup $\left(\overline{M}_t\right)_{t\geq 0}$ to this generator which verifies $ \overline{M}_{s,t}=\overline{M}_{t-s} $ for all $t \geq s$.  We clearly have for all $f \geq 0$ and for all $t\geq 0$ \begin{align*}
 \mathcal{L}_tf(x) \geq  \overline{\mathcal{L}}f(x) \iff M_{s,t}f(x) \geq \overline{M}_{s,t}f(x) \iff M_{s,t}f(x) \geq \overline{M}_{t-s}f(x),
\end{align*}  where we used the relation  $ M_{s,t}\mathcal{L}_tf=\partial_tM_{s,t}f $. \begin{align*}
   \overline{\mathcal{L}} \psi_0(x) &= -4 \frac{x}{x_0^2}(1-(\frac{x}{x_0})^2)\mathbb{1}_{(-x_0,x_0)}(x)+a(t,x)\psi_0(x)+\int_{-x_0}^{x_0} \left[  1- \left( \frac{y}{x_0} \right)^2 \right]^2   Q(x,dy) \\
    & \geq - \frac{4}{x_0}(1-(\frac{x}{x_0})^2)\mathbb{1}_{(-x_0,x_0)}(x)+  \left(\underset{(-x_0,x_0)}{\inf } \underline{a} \right)\psi_0(x)+ \kappa_0 x_0\int_{-1}^{1}(1-y^2)^2 \mathbb{1}_{(x-\epsilon,x+\epsilon)}(x_0 y)  dy  \\
    & \geq -\frac{4}{x_0}(1-(\frac{x}{x_0})^2)\mathbb{1}_{(-x_0,x_0)}(x)+   \left(\underset{(-x_0,x_0)}{\inf } \underline{a} \right)\psi_0(x)+ \kappa_0 \mathbb{1}_{(-x_0,x_0)} x_0\int_{1-\frac{\epsilon}{x_0}}^{1}(1-y)^2(1+y)^2dy  \\
    & \geq - \frac{4}{x_0}(1-(\frac{x}{x_0})^2)\mathbb{1}_{(-x_0,x_0)}(x)+   \left(\underset{(-x_0,x_0)}{\inf } \underline{a} \right)\psi_0(x)+ \kappa_0 \mathbb{1}_{(-x_0,x_0)} x_0 \left( \frac{\epsilon^3}{x_0^3}\left( \frac{1}{5} \left(\frac{\epsilon}{x_0} \right)^2 - \frac{\epsilon}{x_0} - \frac{4}{5} \right) \right) \\
    & \geq - \frac{4}{x_0}(1-(\frac{x}{x_0})^2)\mathbb{1}_{(-x_0,x_0)}(x)+  \left(\underset{(-x_0,x_0)}{\inf } \underline{a} \right)\psi_0(x)+ \frac{8\kappa_0 \epsilon^3}{15 x_0^2} \mathbb{1}_{(-x_0,x_0)}. 
\end{align*} If $2 \kappa_0 \epsilon^3 \geq 15 x_0$, then we have \begin{align*}
\overline{\mathcal{L}}\psi_0(x)  \geq  \frac{4}{x_0}  \left( \frac{x}{x_0} \right)^2 +\left(\underset{(-x_0,x_0)}{\inf } \underline{a} \right)\psi_0(x) \geq \left(\underset{(-x_0,x_0)}{\inf } \underline{a} \right)\psi_0(x).
\end{align*}  Now we suppose that $2 \kappa_0 \epsilon^3 < 15 x_0 $, for $|x | \geq \sqrt{1-\frac{2\kappa_0 \epsilon^3}{15x_0}}x_0$, we clearly have \begin{align*}
\overline{\mathcal{L}}\psi_0(x)  \geq \left(\underset{(-x_0,x_0)}{\inf } \underline{a} \right)\psi_0(x),
\end{align*} and for $|x | < \sqrt{1-\frac{2\kappa_0 \epsilon^3}{15x_0}}x_0$, we have $\sqrt{\psi_0(x)} \geq \frac{2\kappa_0 \epsilon^3}{15 x_0}$ which yields $-\frac{4}{x_0} \sqrt{\psi_0(x)} \geq -\frac{30}{\kappa_0 \epsilon^3}$, from which we deduce that \begin{align*}
\overline{\mathcal{L}}\psi_0(x)  \geq \left(-\frac{30}{\kappa_0 \epsilon^3}+\underset{(-x_0,x_0)}{\inf } \underline{a} +\frac{8 \kappa_0 \epsilon^3}{15 x_0^2} \right) \psi_0(x). 
\end{align*} So for $\beta_0=-\frac{30}{\kappa_0 \epsilon^3}+\underset{(-x_0,x_0)}{\inf } \underline{a}  $, we obtain that $   \overline{\mathcal{L}} \psi_0(x)  \geq \beta_0 \psi_0(x) $. With the relation $\partial_t \overline{M}_{t} \psi_0 =\overline{M}_{t}  \overline{\mathcal{L}}\psi_0$ and the positivity of the propagator, we obtain $\partial_t \overline{M}_{t}  \psi_0 \geq \beta_0 \overline{M}_{t}  \psi_0$, and we deduce that, from the Grönwall's lemma, \begin{align*}
 \overline{M}_{t} \psi_0 \geq e^{\beta_0t} \psi_0. 
\end{align*} We  set  $\overline{\psi}=\overline{M}_T\psi_0$  to have that for all $t  \geq u \geq 0$\begin{align*}
    M_{u,u+t} \overline{\psi} \geq \overline{M}_{t}\overline{M}_T \psi_0 \geq e^{\beta_0 t}  \overline{\psi} .
\end{align*} We note that, because of Lemma \ref{lemposiv},  $\overline{\psi}>0$ . \\

\textbf{Assumption (B1)}  Let us take $s \in (0,T)$,
for $V=1$, we have $ \mathcal{L}_s V(x) \leq \overline{a}(x)V(x)+ \widehat{Q}  $.  Let us  take $r_0>0$ such that for all $|x| \geq r_0$ we have $ \overline{a}(x) \leq -\widehat{Q}+\beta_0-1:=-\widehat{Q} + \alpha_0$. So we have \begin{align*}
\mathcal{L}_s V (x) \leq \alpha_0 V(x)+ (A-\alpha_0+\widehat{Q})\mathbb{1}_{(-r_0,r_0)}(x).
\end{align*} So, if we chose $x_0 \geq 2 \sqrt{r_0}$ in the definition of $\psi_0$ we get $\mathbb{1}_{(-r_0,r_0)} \leq 4\psi_0$, so we finally obtain \begin{align*}
\mathcal{L}_s V (x) \leq \alpha_0 V(x)+\theta_0 \psi_0(x),  
\end{align*} with $\theta_0=4(A+\widehat{Q)}$. For the function $ \phi=V-\theta_0\psi_0$, we obtain $  \mathcal{L}_s \phi \leq  \alpha_0 \phi$ which gives $ M_{s,s+T}\phi \leq e^{\alpha_0 T} \phi $. We finally obtain \begin{align*}
 M_{s,s+T} V \leq e^{\alpha_0 T} V + \theta_0 M_{s,s+T}\psi_0 = \frac{1+e^{\alpha_0T}}{2}V+ \left(\theta_0 M_{s,s+T}\psi_0-\frac{1-e^{\alpha_0T}}{2}V \right).
\end{align*} With the Duhamel formula (\ref{2.5})  we have \begin{align*}
 M_{s,s+T}\psi(x)=\psi(x+T -s)e^{ \int_{s}^{s+T}{a(\tau,x+\tau-s)d\tau}} + \int_{s}^{s+T}{ e^{\int_{s}^{\tau}{a(\tau',x+\tau'-s)d\tau'}} \int_{\mathbb{R}}^{}{M_{\tau,s+T }\psi(y)Q(x+ \tau-s,dy)}d\tau}. 
\end{align*}  Since, $ \mathcal{L}_sV \leq \overline{a}+\widehat{Q}$ we obtain  for all $\tau \in [s,s+T]$ 
\begin{align} \label{petit}
M_{\tau,s+T }\psi_0(x) \leq M_{\tau, s+T } V(x) \leq e^{(\overline{a}+\widehat{Q})(s+T -\tau)}V  \leq e^{(\overline{a}+\widehat{Q}) V} := \delta V,
\end{align} so, we have \begin{align*}
 M_{s,s+T}\psi_0(x) \leq \psi_0(x+t-s)e^{ \int_{s}^{s+T}{a(\tau,x+\tau-s)d\tau}} + \widehat{Q}\int_{s}^{s+T}{ e^{\int_{s}^{\tau}{a(\tau',x+\tau'-s)d\tau'}}e^{(\overline{a}+\widehat{Q})(t-\tau)}d\tau}. 
\end{align*} by dominated convergence we have $ \lim\limits_{x \rightarrow +\infty} \frac{M_{s,s+T}\psi_0(x)}{V(x)} = 0. $  So we can find $C>0$ and $R>r_0$ such that \begin{align*}
 \theta_0 M_{s,s+T}\psi_0-\frac{1-e^{\alpha_0T}}{2}V \leq  C \mathbb{1}_{[-R,R]},
\end{align*} which gives (B1)  with $K=(-R,R)$. \\

 \textbf{Assumption (B0) } We take $\psi=\delta^{-1}\overline{\psi}$ which preserves (B1) and (B2) and gives the relation $\psi \leq V$ from (\ref{petit}).  \\

\textbf{Assumption (B4)} Let us take $s\geq 0$ and $u \in [s,s+T]$.
We want to prove that $  M_{s,s+kT}\psi  \lesssim   M_{u,+u+kT}\psi $ in  $K=(-R,R)$ with $r_0>0$ defined at the end of the proof of Assumption~(B1) above. From Lemma \ref{lemma imp} we have $M_{s,s+kT}\psi \lesssim  M_{s,u+kT}\psi $ on $K$, so it is sufficient to prove that $M_{s,u+kT}\psi  \lesssim M_{u,u+kT}\psi$ on $K$.  

\begin{align*}
    M_{s,u+kT}\psi(x) &= \psi(x+kT+u-s)e^{\int_{s}^{u+kT}a(\tau,x+\tau-s) d\tau}+\int_{s}^{u+kT}e^{\int_{s}^{\tau}a(\tau',x+\tau'-s) d\tau'} \int_{\mathbb{R}}^{}M_{\tau,u+kT}\psi(y) Q(x+\tau-s,dy)d\tau \\
    &\leq \psi(x+kT) e^{\int_{u}^{u+kT+(u+s)}a(\tau+s-u,x+\tau-u) d\tau}\\
    &\hspace{1cm} + C_1 \int_{s}^{u+kT}e^{\int_{s}^{\tau}a(\tau',x+\tau'-s) d\tau'} \int_{\mathbb{R}}^{}M_{\tau,u+kT}\psi(y) Q(x+\tau-u,dy)d\tau \\
    & \leq C e^{AT}\psi(x+kT) e^{\int_{u}^{u+kT}a(\tau,x+\tau-u) d\tau}\\
    &\hspace{1cm} + C_1\int_{s}^{u+kT}e^{\int_{s}^{u}a(\tau',x+\tau'-s) d\tau'} e^{\int_{u}^{\tau}a(\tau',x+\tau'-s) d\tau'} \int_{\mathbb{R}}^{}M_{\tau,u+kT}\psi(y) Q(x+\tau-u,dy)d\tau \\
    &\leq C e^{AT}\psi(x+kT) e^{\int_{u}^{u+kT}a(\tau,x+\tau-u) d\tau}\\
    &\hspace{1cm} + C_1 e^{A T}  \int_{s}^{u+kT}e^{\int_{u}^{\tau+u-s}a(\tau'-u+s,x+\tau'-u) d\tau'} \int_{\mathbb{R}}^{}M_{\tau,u+kT}\psi(y) Q(x+\tau-u,dy)d\tau \\
    &\leq C e^{AT}\psi(x+kT) e^{\int_{u}^{u+kT}a(\tau,x+\tau-u) d\tau}\\*
    &\hspace{1cm} + C_1 C e^{A T} \int_{s}^{u+kT}e^{\int_{u}^{\tau+u-s}a(\tau',x+\tau'-u) d\tau'} \int_{\mathbb{R}}^{}M_{\tau,u+kT}\psi(y) Q(x+\tau-u,dy)d\tau \\
    &\leq C e^{AT}\psi(x+kT)  e^{\int_{u}^{u+kT}a(\tau,x+\tau-u) d\tau}\\
    &\hspace{1cm} + C_1C e^{2A T} \int_{s}^{u+kT}e^{\int_{u}^{\tau}a(\tau',x+\tau'-u) d\tau'} \int_{\mathbb{R}}^{}M_{\tau,u+kT}\psi(y) Q(x+\tau-u,dy)d\tau.
\end{align*} So, to conclude the proof we only need to check that 
\begin{align*}
    I&:=\int_{s}^{u}e^{\int_{u}^{\tau}a(\tau',x+\tau'-u) d\tau'} \int_{\mathbb{R}}^{}M_{\tau,u+kT}\psi(y) Q(x+\tau-u,dy)d\tau \\
    & \leq \int_{u}^{u+(k+1)T}e^{\int_{u}^{\tau}a(\tau',x+\tau'-u) d\tau'} \int_{\mathbb{R}}^{}M_{\tau,u+(k+1)T}\psi(y) Q(x+\tau-u,dy)d\tau.
\end{align*} First, we note that by continuity of $a$, there exist $\delta>0$ such that, for all $\tau \in (s+T,u+T ) $, we have  $e^{\int_{\tau-T}^{\tau}a(\tau',x+\tau'-u)d\tau'} \geq \delta $.
With the change of variable $\tau \leftarrow \tau+T$ we obtain

\begin{align*}                   
   I &= \int_{s+T}^{u+T}e^{\int_{u}^{\tau-T}a(\tau',x+\tau'-u) d\tau'} \int_{\mathbb{R}}^{}M_{\tau,u+(k+1)T}\psi(y) Q(x+\tau-T-u,dy)d\tau \\ 
   &\leq \frac{1}{\delta}C_1\int_{u}^{u+(k+1)T}e^{\int_{u}^{\tau}a(\tau',x+\tau'-u) d\tau'} \int_{\mathbb{R}}^{}M_{\tau,u+(k+1)T}\psi(y) Q(x+\tau-u,dy)d\tau. \\
\end{align*} So we set $\tilde{C}=\max\left(\frac{Ce^{AT}}{2},C_1Ce^{2AT},\frac{1}{\delta}C_1 \right)$ and we finally get \begin{align*}
    M_{s,u+kT}\psi(x) \leq \tilde{C} \left( M_{u,u+kT}\psi(x)+M_{u,u+(k+1)T} \psi(x) \right). 
\end{align*} We obtain the result because in $K$, we have $M_{u+(k+1)T}\psi \lesssim M_{u,u+kT}\psi  $. \\

\end{proof}

\section{ Application to the growth-fragmentation equation}

\vspace{0.1cm}

We apply our method to the following non-local equation
\begin{equation}\label{gro}
\left\{\begin{array}{l}
    \displaystyle\partial_t u(t,x) +   \partial_x (g(t,x) u(t,x)) +\beta(t,x)u(t,x) =\int_{x}^{\infty} b(t,y,x) u(t,y)dy  , \quad (t,x)\in (s,\infty) \times \mathbb{R}_+, \vspace{2mm}\\
    u(s,x)=u_s(x),
\end{array}\right.
\end{equation}
where the fragmentation kernel is of the form \begin{align*}
 b(t,x,y)= \frac{1}{x}\kappa\left(\frac{y}{x} \right)\beta(t,x),   
\end{align*} where $\kappa$ is the fragmentation distribution which verifies $\kappa(\cdot) \geq \underline{\kappa} >0$ for $\underline{\kappa}>0 $ and is continuous. For any $k\geq0$, we set \begin{align*}
 \eta_k=  \int_{0}^{1}z^k \kappa(z)dz. 
\end{align*} We note that the conservation of mass during the fragmentation leads to impose $\eta_1=1$, so we have $\eta_0>1$. We also note that theses assumptions on $\kappa$ leads that\begin{align}\label{major}
    \forall \alpha>1,\qquad \eta_\alpha <1 .
 \end{align}
 Indeed, $ \eta_\alpha= 1+\int_{0}^{1} (z^\alpha-z)\kappa(z)dz \leq 1+\underline{\kappa}\int_{0}^{1} (z^\alpha-z)dz =1+\frac{1-\alpha}{2(\alpha+1)}\underline{\kappa} <1 $.
In this section we detail an example for illustrative but non-exhaustive purposes, we study the case $g(t,x)=g_0(t)+g_1(t)x $ where $g_0,g_1$ are $T$-periodic continuous functions and $\beta(t,x)= \beta_0(t)+\beta_1(t) x$ where $\beta_0,\beta_1$ are $T$-periodic continuous functions. We also assume that there exists $c>0$ such that  $g_0(t), \beta_1(t) \geq c. $

\begin{monTheoreme}\label{th:croiss-frag}  Let us take $V=1+x^\alpha$ with $\alpha>1$. 
There exist constants $C,\omega > 0$ and a unique T-periodic Floquet family   $(\gamma_{s} , h_{s}, \lambda_F)_{s\leq t} \subset \mathcal{M}(V) \times \mathcal{B}(V) \times \mathbb{R}$ with $\langle \gamma_{s}, h_{s} \rangle =1$ for all $s \geq 0$ and $ \|h_{s_0} \|_ {\infty}=1$ for a certain $s_0 \geq 0$,  such that for any initial condition $u_s \in L^1(\mathbb{R}_+,dx)$ the corresponding solution $u(t,x)$ of equation (\ref{gro}) verifies for all $t \geq s$

$$ \left\| u(t,.)e^{-\lambda_F(t-s)}-\left(\int_{0}^{\infty}{h_{s,s}(x) u_s(x)dx}\right)\gamma_{s,t} \right\|_{L^1(\mathbb{R}_+)} \leq C \left\|  u_s \right\|_{L^1(\mathbb{R}_+)} e^{-\omega (t-s)}.$$
\end{monTheoreme}

\pg{ This result allows, in particular, to complete  a result of~\cite{gabriel2011long} which states that, for $g_1=\beta_0=0$, $\beta_1$ independent of $t$, and $\kappa$ symmetric, we have the comparison
\[\overline{\lambda(g_0)}\leq\lambda_F\leq \lambda(\overline{g_0}),\]
where, for a function $a(t)$ $T$-periodic we use the notation
\[\overline a = \frac1T\int_0^T a(s)ds,\]
and, for any positive constant $g_0$, $\lambda(g_0)$ is the principal eigenvalue of the operator
\[\mathcal G u (x) = -g_0 u'(x) - \beta_1 x u(x) + \beta_1 \int_x^\infty u(y)\kappa\Big(\frac xy\Big)dy.\]
Indeed, this result is proved in~\cite[Proposition~6.2]{gabriel2011long} under the condition that the Floquet eigenelements exist, but the author explains that his method does not allow him to ensure this existence.}

\

\pg{Similarly as in Section 2, we associate to the PDE a semiflow through the Duhamel formula}
\begin{align}\label{eq:Duhamel_frag}
    M_{s,t}f(x)=f\left(X_{s,t}(x) \right) & \, e^{-\int_{s}^{t} \beta(\tau,X_{s,\tau}(x))d\tau } \\
    & + \displaystyle \int_{s}^{t}e^{-\int_{s}^{u} \beta(\tau',X_{s,\tau'}(x))d\tau' } \beta(\tau,X_{s,\tau}(x)) \displaystyle\int_{0}^{1}M_{\tau,t}f(zX_{s,\tau}(x))\kappa(z)dzd\tau \nonumber
    \end{align} 
    where \begin{align*}
    X_{s,t}(x)=xe^{\int_{s}^{t}g_1}+\int_{s}^{t}g_0\pg{(\tau)e^{\int_s^\tau g_1}d\tau},
    \end{align*} is the solution to the characteristic equation \begin{align*}
\left\{
\begin{array}{rcr}
\partial_s X_{s,t}(x) =& -g(s,X_{s,t}(x))\\
X_{t,t}(x) =&  x.
\end{array}
\right.
\end{align*} \pg{The construction of this semiflow through a fixed-point argument is very similar to the construction in Section~2, and we do not repeat it for the sake of conciseness.} We \pg{also} refer to \cite{bansaye2019non} for such a construction but with coefficients independent of time.

\

\pg{For proving Theorem~\ref{th:croiss-frag}}, we will apply Theorem \ref{thm1.2}, by finding directly the eigenvectors of the generator $\mathcal{L}_t$ defined by \begin{align*}
     (\mathcal{L}_t f)(x)&= g(t,x)f'(x)+\beta(t,x)\displaystyle \int_{0}^{1}\kappa\left( z\right)f(zx)dz-\beta(t,x)f(x) 
\end{align*}which verifies the relation $\partial_s M_{s,t}f= - \mathcal{L}_s M_{s,t}f$ for all $f \in \mathcal{C}^1 \cap \mathcal{B}(V)$ with $V(x)=1+x^\alpha$ for a certain $ \alpha>1$. More precisely, we search a family $(h_t)_{t \geq 0 } $ of $ \mathcal{C}^1 $ functions, and a constant $\lambda_F>0$   which verify $h_s=e^{\lambda_F(s-t)}M_{s,t}h_t$. This equality gives  \begin{align*}
\partial_s h_s= (\lambda_F- \mathcal{L}_s)h_s. 
\end{align*}

\pg{
\begin{lemma}
There exist $\lambda_F\in\mathbb R$ and $t\mapsto u_t, t\mapsto v_t$ two $T$-periodic functions with  $\mathcal C^1$ regularity such that
\[h_s(x):=u_{-s}+v_{-s} x\qquad\text{satisfies}\qquad h_s=e^{\lambda_F(s-t)}M_{s,t}h_t.\]
\end{lemma}}

\begin{proof}
We look for $\lambda_F\in\mathbb R$ and periodic functions $u_t$ and $v_t$ such that $h_t$ defined by $h_t(x)=u_{-t}+v_{-t}x$ satisfies $\partial_s h_s= (\lambda_F- \mathcal{L}_s)h_s$.
We start by looking for particular solutions to the equation \begin{align*}
\partial_s \varphi_s= - \mathcal{L}_s\varphi_s
\end{align*} under the form $\varphi_t(x)= m_{-t}+ n_{-t}x$. Such solutions need to verify the following ODE system

 \begin{align}\label{edo}
\begin{pmatrix} 
         \dot m_t \\ 
         \dot n_t
   \end{pmatrix} = \begin{pmatrix} 
          \beta_0(-t) (\eta_0-1)&g_0(-t) \\ 
         \beta_1(-t) (\eta_0-1)&g_1(-t)
   \end{pmatrix}  \cdot \begin{pmatrix} 
         m_t\\ 
         n_t
   \end{pmatrix} := A(t) \begin{pmatrix} 
         m_t\\ 
         n_t
   \end{pmatrix}. 
    \end{align}
    Let us define the flow mapping \begin{align*}
    \Xi_t:  
   \begin{pmatrix} 
         n_0 \\ m_0
   \end{pmatrix} \longrightarrow \begin{pmatrix} 
         n_t \\ m_t
   \end{pmatrix}, 
    \end{align*} where $n_t,m_t  $ are solution of (\ref{edo}).
  \pg{If we can prove that $\Xi_T$  is a strictly positive matrix, then the Perron-Frobenius theorem ensures that} $\Xi_T$ admits  $\Lambda>0$ and $\begin{pmatrix} 
         u_0 \\ v_0
   \end{pmatrix} > 0$ such that \begin{align*}
    \Xi_T \begin{pmatrix} 
         u_0 \\ v_0
   \end{pmatrix} = \Lambda\begin{pmatrix} 
         u_0 \\ v_0
   \end{pmatrix},  
   \end{align*} \pg{and we easily check that the result follows by setting
   \[\lambda_F=\frac{\log\Lambda}{T}\qquad \text{and}\qquad 
   \begin{pmatrix} 
         u_t\\ 
         v_t
   \end{pmatrix}
   = e^{-\lambda_F t}\,\Xi_t  \begin{pmatrix} 
         u_0 \\ v_0
   \end{pmatrix}.\]}
   
   Let us then prove that $\Xi_t$ has strictly positive entries for any $t>0$.
   Consider that $m_0>0$ and $n_0 \geq 0$.
   By continuity of $ t \mapsto \begin{pmatrix} 
         n_t \\ m_t
   \end{pmatrix}  $, we have that there exists $ \epsilon>0$ such that $m_t>0$ for $t \in [0,\epsilon)$. Let us set  $w_t= e^{-\int_{0}^{t}g_1(-u)du}n_t,$  we have \begin{align*}
   \dot w_t=\beta_1(-t)(\eta_0-1)e^{-\int_{0}^{t}g_1(-u)du}m_t \geq c  (\eta_0-1)e^{-\int_{0}^{t}g_1(-u)du}m_t, 
\end{align*} and we obtain that  $\dot w_t >0 $ in $[0,\epsilon)$. We deduce that $n_t >0$ for any $t \in (0,\epsilon)$. Let us consider the set \begin{align*}
 X=  \big\{  T>0 \ | \ n_s>0, \ m_s>0, \ \forall s \in (0,T)  \big\}.
\end{align*} We clearly have $\epsilon \in X$. Let us suppose that $\sup X<\infty$, so there exist $T^*>0$ such that $ n_s>0$ and $m_s>0$ for any $ s \in (0,T^*)$.
As the functions are increasing we clearly have $u_{T^*},v_{T^*}>0$ and by continuity, there exist $\delta>0$ such that $u_s>0$ and $v_s>0$ for all  $s\in (0,T^*+\delta)$, and we obtain that $\sup X=\infty$, so $\Xi_t$ is positive. 
\end{proof}

We will now prove that the operator $M_{s,s+T}$  verifies Assumption A with  $\psi=h_0$ and $V(x)=1+x^\alpha$, $\alpha>1$.

\begin{proof}[Proof \pg{of Theorem~\ref{th:croiss-frag}}]
The proof consists in verifying Assumption~A and then apply Theorem~\ref{thm1.2} to conclude.

\medskip

\textbf{Assumption (A1)} 
We compute \begin{align*}
    \mathcal{L}_sV(x) &= x^{\alpha} \left( \frac{\alpha\tau_0(s)}{x}+\alpha\tau_1(s)+\frac{\beta_0(s)(\eta_\alpha-1)}{x^\alpha}  +\beta_1(s)x(\eta_0-1)\right) \\
    &\leq  x^{\alpha} \left( \frac{\alpha\tau_0(s)}{x}+\alpha\tau_1(s)+\frac{\beta_0(s)(\eta_0-1)}{x^\alpha} +c x(\eta_\alpha-1)\right). 
\end{align*}  We clearly have $\lim\limits_{x \rightarrow +\infty} \frac{\alpha\tau_0(s)}{x}+\alpha\tau_1(s)+\frac{\beta_0(s)(\eta_0-1)}{x^\alpha} +c x(\eta_\alpha-1)=-\infty $, since $\eta_\alpha<1$, so we can find $x_1>0$ and $ \delta<0$ such that \begin{align*}
\forall x \geq x_1, \quad  \mathcal{L}_sV(x) \leq \delta x^\alpha = \delta V(x) - \delta.
\end{align*}  By continuity of the function $\psi$, we deduce that there exist $\theta_0>0 $ and $R>0$ such that  \begin{align*}
 \forall x \geq 0, \quad  \mathcal{L}_sV(x) \leq \delta V(x) +\theta_0 \mathbb{1}_{K}  \psi
\end{align*} where $K=[0,R]$.  For the function $\phi=V-\theta_0\mathbb{1}_{[0,R]} \psi$, we obtain $ \mathcal{L}_s \phi \leq \delta \phi$ which gives $M_{s,s+T} \phi \leq e^{\delta T} \phi$. We finally obtain \begin{align*}
 M_{s,s+T}V \leq e^{\delta T}V + \theta_0 M_{s,s+T}\left(\mathbb{1}_{[0,R]} \psi \right).
\end{align*} By construction of $\psi=h_0$ we have $M_{s,s+T}\psi= e^{\lambda_FT}\psi$ and in particular $M_{s,s+T}\left(\mathbb{1}_{[0,R]} \psi \right)= e^{\lambda_FT} \mathbb{1}_{[0,R]}\psi$, which finally gives $  M_{s,s+T}V \leq e^{\delta T}V+e^{\lambda_F T}\theta_0 \mathbb{1}_{[0,R]} \psi. $ So we get (A1) with $ \alpha = e^{\delta T} < 1 $ and $\theta=e^{\lambda_F T}\theta_0$. \\

\textbf{Assumption (A2)} By construction of  $h_0$ we have  $M_{s_0,s_0+T}\psi= e^{\lambda_FT}\psi$, so we get (A2) with $\beta=e^{\lambda_FT} > 1 > \alpha$.  \\

\textbf{Assumption (A3)} 
We  would like to prove the Doeblin condition that is, for $s \geq 0$ and a compact $\mathcal{C}$, there exist  $k \in \mathbb{N}$, and a measure  $\nu \in \mathcal{P}(\mathcal{C}) $ such that \begin{align*}
M_{s,s+kT}f \geq c \langle \nu, f \rangle \mathbb{1}_\mathcal{C}.
\end{align*} Let us take $s \in \mathbb{R}_+$ and  $\mathcal{C}=[0,R] $ with $R$ defined on $(A1)$. For all $\tau \in [s,t]$ we have from~\eqref{eq:Duhamel_frag} \begin{align*}
M_{\tau,t} f(z X_{s,\tau}(x)) \geq f\left(  X_{\tau,t}( z X_{s,\tau}(x) )(x) \right) e^{-\int_{\tau}^{t} \beta  \left( \tau',  z X_{s,\tau'}(x) \right)d\tau'    } \geq  f\left(  X_{\tau,t}( z X_{s,\tau}(x) )(x) \right)  e^{-\int_{s}^{t} \beta_0(\tau')+\beta_1(\tau')zX_{s,\tau'}(x)d\tau' }. 
\end{align*}  Functions $g_0,g_1,\beta_0,\beta_1$ are continuous and periodic, so they are bounded and we deduce that there exists  $B \in \mathbb{R}$  such that  for all $x \in \mathcal{C}$ and $\tau' \in [s,t]$ we have $   \beta  \left( \tau',  z X_{s,\tau'}(x) \right)\leq B$. We deduce that for all $x \in \mathcal{C}$ we have $ e^{-\int_{s}^{t}  \beta_0(\tau')+\beta_1(\tau')zX_{s,\tau'}(x)d\tau'} \geq  e^{(s-t)B}  $. We obtain by using the Duhamel formula \eqref{eq:Duhamel_frag} again that \begin{align*}
    M_{s,t}f(x) &\geq  e^{(s-t)B} \underline{\kappa} c \int_{s}^{t}   \int_{0}^{1} f\left(  X_{\tau,t}( z X_{s,\tau}(x) )(x) \right) dz d\tau.
\end{align*} We have $  \partial_z X_{\tau,t}( z X_{s,\tau}(x) )(x) = X_{s,\tau}(x)e^{\int_{\tau}^{t}\tau_1} \leq X_{s,\tau}(R)  $. We note that there exists $a_1>0$ and $a_2>0$ such that for all $s \leq t_1 \leq t_2 \leq t$ we have, $ a_1  X_{s,t} \leq X_{t_1,t_2} \leq a_2X_{s,t} $.  By   the change of variable $u=X_{\tau,t}( z X_{s,\tau}(x) )(x)$ we obtain

\begin{align*}
    M_{s,t}f(x) & \geq    e^{(s-t)B} \underline{\kappa} c \int_{s}^{t} \int_{X_{\tau,t}(0)}^{X_{s,t}(x)} f(z) \frac{1}{a_2 X_{s,t}(R)}dzd\tau  \\
     & \geq  e^{(s-t)B} \underline{\kappa} c \int_{s}^{t} \int_{a_1X_{s,t}(0)}^{X_{s,t}(x)} f(z) \frac{1}{a_2 X_{s,t}(R)}dzd\tau   \\
     &\geq \frac{t-s}{a_2 X_{s,t}(R)} e^{(s-t)B} \underline{\kappa} c  \int_{a_1 X_{s,t}(0)}^{X_{s,t}(x)} f(z) dzd\tau  
\end{align*} So we proved \begin{align*}
  M_{s,t}f(x) \geq c_{s,t} \langle \nu_{s,t} ,f \rangle, 
\end{align*} with $c_{s,t} =   \frac{t-s}{a_2 X_{s,t}(R)} e^{(s-t)B} \underline{\kappa} c $ and $\nu_{s,t}(dx)= \mathbb{1}_{a_1X_{s,t}(0)<x< X_{s,t}(0)}dx$.\\

\textbf{Assumption (A4)} For all $n \in \mathbb{N}$ $\langle \nu, \frac{M_{s,s+nT}\psi}{\psi} \rangle  = e^{n\lambda_F T}=\frac{M_{s,s+nT}\psi}{\psi} $ which gives (A4) with $d=1$. 

\medskip

As announced at the beginning of the proof, the conclusion follows from applying Theorem~\ref{thm1.2}.
\end{proof}

\section{Conclusion and perspectives}

In this paper, we have extended the contraction approach performed in~\cite{bansaye2020ergodic} to the Harris setting and it allowed us to derive new results about the Malthusian asymptotic behavior of non-conservative and non-local periodic equations set on an unbounded domain, namely mutation-selection and growth-fragmentation models.
However, the method we developed has its limitations.
For verifying the condition~(B5), we strongly use the fact that the equations are set on one dimensional domains.
Indeed, in dimension one, the drift term guarantees that any position can be reached with positive probability.
Extending the results to higher dimensional selection-mutations models is left for future works.
Besides, in the case of the selection-mutations model, we were not able to make the mutation kernel depend on time. Considering such time-periodic mutations is also an interesting perspective of future work.

\medskip

Another interesting issue is the extension to the time periodic setting of the works~\cite{Henry2022,Patout2022,Henry2023} which aim at understanding the backward dynamics of the ancestral lineages in (one-dimensional) selection-mutations models.
The equations considered in~\cite{Patout2022,Henry2023} correspond exactly to Equation~\eqref{yop} in a time-homogeneous environment, and the authors take advantage of the result obtain in~\cite{cloez2020irreducibility}.
This makes me think that the methodology developed in the present paper could be used for extending these results to the case with time-periodic coefficients.
The papers~\cite{Henry2022,Patout2022,Henry2023} also provide some numerical illustrations of their theoretical results.
Numerical simulations in periodic environments would definitely be of great interest for understanding further the underlying dynamics of the individuals in the population.

\section*{Acknowledgements}
The author is very grateful to Bertrand Cloez and Pierre Gabriel for all the discussions on the subject and useful help that allowed improving the paper. He also acknowledges the support from the ANR project NOLO (ANR-20-CE40-0015), funded by the French Ministry of Research, and the funding of his PhD grant by the foundation Jacques Hadamard.

\bibliographystyle{abbrv}
\bibliography{b}

\end{document}